\documentclass[11pt,reqno]{amsart}
\usepackage{amsmath,amssymb}
\usepackage{mathabx}
\usepackage{mathdots}
\usepackage{graphicx,caption,subfig}
\usepackage{color,soul}
\usepackage{hyperref}
\usepackage{enumerate}
\usepackage[colorinlistoftodos]{todonotes}
\usepackage{enumitem}
\usepackage[normalem]{ulem}
\usepackage{constants}

\newconstantfamily{N}{symbol=N}

\usepackage{fullpage}

\newtheorem{theorem}{Theorem}[section]
\newtheorem{lemma}{Lemma}[section]
\newtheorem{definition}{Definition}[section]
\newtheorem{corollary}{Corollary}[section]
\newtheorem{proposition}{Proposition}[section]
\newtheorem{remark}{Remark}[section]

\newcommand{\tab}{\hspace*{2em}}

\numberwithin{equation}{section}

\DeclareMathOperator{\Arg}{Arg}

\DeclareMathOperator{\diag}{diag}

\DeclareMathOperator{\Span}{span}

\newcommand{\CC}{\mathbb{C}}
\newcommand{\NN}{\mathbb{N}}
\newcommand{\RR}{\mathbb{R}}
\newcommand{\cj}{\cv^{(j)}}

\newcommand{\dt}{\textrm{\normalfont d}}

\renewcommand{\vec}[1]{\boldsymbol{\mathrm{#1}}}
\newcommand{\xv}{\mathcal{X}} 
\newcommand{\yv}{\mathcal{Y}} 
\newcommand{\xvec}{\vec{x}}
\newcommand{\cv}{\mathcal{C}} 
\newcommand{\vv}{\vec{v}}
\newcommand{\wv}{\vec{w}}
\newcommand{\WB}{\mathcal{W}}
\newcommand{\uv}{\vec{u}}
\newcommand{\UB}{\mathcal{U}}
\newcommand{\zv}{\vec{z}}

\newcommand{\av}{\vec{a}}
\newcommand{\bv}{\vec{b}}
\newcommand{\fv}{\vec{f}}

\newcommand{\Lx}{L({\xv})}

\newcommand{\VV}{\mathbf{V}}
\newcommand{\WW}{\mathbf{W}}
\newcommand{\UU}{\mathbf{U}}
\newcommand{\EE}{\mathbf{E}}
\newcommand{\HilbM}{\bar{\mathbf{H}}} 
\newcommand{\AAA}{\mathbf{A}}
\newcommand{\BBB}{\mathbf{B}}
\newcommand{\CCC}{\mathbf{C}}
\newcommand{\QQ}{\mathbf{Q}}
\newcommand{\RRR}{\mathbf{R}}
\newcommand{\II}{\mathbf{I}}
\newcommand{\GG}{\mathbf{G}}
\newcommand{\VVV}{\widetilde{\mathbf{V}}}
\newcommand{\DD}{\boldsymbol{\mathcal{D}}}
\newcommand{\PP}{\mathbf{P}}
\newcommand{\PM}[1][m]{\mathbf{Q}_{#1}} 
\newcommand{\MS}{\mathcal{M}}
\newcommand{\MO}{\mathcal{O}}

\newcommand{\sep}{\eta}

\makeatletter
\DeclareRobustCommand*\cal{\@fontswitch\relax\mathcal}
\makeatother

\begin{document}
\title[Vandermonde matrices with clustered nodes] {The spectral
  properties of Vandermonde matrices with clustered nodes}

\author[D.Batenkov]{Dmitry Batenkov} \address{Department of Applied
Mathematics, School of Mathematical Sciences, Tel Aviv University,
P.O. Box 39040, Tel Aviv 6997801, Israel}
\email{dbatenkov@tauex.tau.ac.il} \thanks{}

\author[B.Diederichs]{Benedikt Diederichs} \address{University of
  Passau and Fraunhofer IIS Research Group Knowledge Based Image
  Processing, Passau, Germany}
\email{benedikt.diederichs@uni-passau.de} \thanks{}

\author[G. Goldman]{Gil Goldman} \address{Department of Mathematics,
  The Weizmann Institute of Science, Rehovot 76100, Israel}
\email{gil.goldman@weizmann.ac.il} \thanks{}

\author[Y. Yomdin]{Yosef Yomdin} \address{Department of Mathematics,
  The Weizmann Institute of Science, Rehovot 76100, Israel}
\email{yosef.yomdin@weizmann.ac.il}

  \date{}

\begin{abstract}
  We study rectangular Vandermonde matrices $\VV$ with $N+1$ rows and
  $s$ irregularly spaced nodes on the unit circle, in cases where
  some of the nodes are ``clustered'' together -- the elements inside each cluster being
  separated by at most $h \lesssim {1\over N}$, and the
  clusters being separated from each other by at least $\theta \gtrsim {1\over N}$. 
  We show that any pair of column subspaces corresponding to two different clusters are nearly orthogonal: 
  the minimal principal angle between them is at most  
  $$\frac{\pi}{2}-\frac{c_1}{N \theta}-c_2 N h,$$ 
  for some constants $c_1,c_2$ depending only on the multiplicities of the
  clusters.
  As a result, spectral analysis of $\VV_N$ is significantly
  simplified by reducing the problem to the analysis of each cluster
  individually.
  Consequently we derive accurate estimates for 1) all the singular values
  of $\VV$, and 2) componentwise condition numbers for the linear
  least squares problem.  Importantly, these estimates are exponential
  only in the local cluster multiplicities, while changing at most	
  linearly with $s$.
\end{abstract}

\maketitle

\section{Introduction}\label{sec.intro}

\subsection{Background}

For an ordered set of distinct nodes $\xv=\{ x_1,\ldots,x_s\}$ with
$x_j \in (-\pi,\pi]$, and $N \geq s-1$, we consider the
$(N+1) \times s$ Vandermonde matrix $\VV=\VV_{N}(\xv)$ with nodes
$\{e^{\imath x_j}\}_{j=1}^s$, given by\footnote{Note a slight abuse of notation as $\VV$ depends not only on
  $\xv$, but also on the ordering of the nodes. Therefore, we always
  assume that the set of nodes comes with an arbitrary, but fixed,
  ordering.  }

\begin{equation}
  \label{eq.vand-def}
  \VV_N(\xv)=
  \begin{bmatrix}
    1 & 1 & \dots & 1 \\
    e^{\imath x_1} & e^{\imath x_2} & \dots & e^{\imath x_s} \\
    e^{\imath 2x_1} & e^{\imath 2x_2} & \dots & e^{\imath 2x_s} \\
    \vdots & \vdots & \vdots & \vdots \\
    e^{\imath Nx_1} & e^{\imath Nx_2} & \dots & e^{\imath Nx_s}
  \end{bmatrix}.
\end{equation}

Square and rectangular Vandermonde matrices have been studied quite
extensively by numerical analysts due to their close relation to
polynomial interpolation and approximation, quadrature and related
topics, see
e.g. \cite{crdova1990,beckermann2000,gautschi_inverses_1962,gautschi_inverses_1963,gautschi_inverses_1978,gautschi_norm_1974,eisinberg2001,beckermann2019,pan2016,tyrtyshnikov1994,beckermann_sensitivity_1999}
and references therein. The matrices $\VV_N$ as in \eqref{eq.vand-def}
have also received recent attention in the applied harmonic analysis
community with relation to the problem of mathematical
super-resolution
\cite{batenkov2019a,batenkov_stability_2016,moitra_super-resolution_2015,batenkov2018a,li2019,li_stable_2017,diederichs2019,demanet_recoverability_2014,kunis2019a,kunis2018},
where the magnitude of their smallest singular value controls the
limit of stable recovery of point sources from bandlimited
data. Similar connections exist in spectral estimation and direction
of arrival problems, where $\VV_N$ are closely related to data
covariance matrices
\cite{lee1992,stoica_spectral_2005,tuncer_classical_2009,yang2018}.

While Vandermonde matrices with real nodes are known to be
ill-conditioned (for instance, the condition number must grow
exponentially in $s$, see
\cite{beckermann_sensitivity_1999,beckermann2019,pan2016} and
references therein), the situation may be drastically different for
complex nodes. Indeed, the columns of $\VV_N$ become orthogonal when
$\xv$ is a subset of the roots of unity of order $N+1$, but on the
other hand may be arbitrary close to each other if two or more nodes
collide.  When the minimal distance\footnote{All distances are in the
  wrap-around sense, to be defined precisely below.} between any two
nodes in $\xv$ (denoted by $\sep$ in this section) is larger than
${1\over N}$, the matrix $\VV_N(\xv)$ is known to be
well-conditioned. The sharpest result in that direction was recently
presented in \cite{diederichs2019}, building upon earlier results
\cite{aubel_vandermonde_2017,moitra_super-resolution_2015,li_stable_2017,bazan_conditioning_2000}. On
the other hand, the singular value decomposition of $\VV_N$ in the
special case of equispaced and nearly colliding nodes (i.e.
$x_j=x_0 + j\sep$ and $N\sep\ll 1$) can be derived from the seminal
works on the spectral concentration problem by Slepian and co-workers,
see \cite{batenkov2018a,slepian_prolate_1978,slepian1968} and
references therein. In this case, $\VV_N$ becomes severely
ill-conditioned, e.g. $\kappa\left(\VV_N\right)\sim (N\sep)^{-s+1}$,
analogous to the situation with real-valued nodes. Between the two
extremes mentioned above, the general case of irregularly spaced and
partially colliding nodes is much less investigated.

\subsection{The partial clustering model}

In the context of super-resolution (see references in
the previous section, and in particular the detailed discussions in
\cite{batenkov2018a,batenkov2019a,demanet_recoverability_2014}), the
phase transition $\sep\approx 1/N$ corresponds to the classical
Rayleigh-Nyquist limit.  In the case $\eta \ll {1\over N}$, it was
shown in e.g. \cite{li_stable_2017} that the error amplification for
recovering a sparse atomic measure supported on $\xv$ from $N$ Fourier
coefficients can be as large as $\approx (N\sep)^{-2s+1}$, and
moreover this worst-case scenario happens precisely when all the nodes
are ``clumped'' together (this is in fact equivalent to Slepian's
equispaced configuration). In applications of super-resolution (see
e.g. \cite{bhandari2019}), frequently there exists a prior information
that only a small number of nodes, say $\ell\leq s$, can become very
close to each other (with respect to the Rayleigh length scale $1/N$),
in which case one can expect much more stable recovery. Indeed, it was
very recently shown in \cite{batenkov2018a} that in this case, the
minimax error rate scales like $(N\eta)^{-2\ell+1}$, albeit with
proportionality constants which decay exponentially in $s$. In
\cite{batenkov2019a} a closely related problem of super-resolution
from continuous frequency measurements in a band
$\left[-\Omega,\Omega\right]$ under the clustered model was
investigated in the sub-Rayleigh regime $\eta\ll{1\over\Omega}$, and
the error rate $(\Omega \eta)^{-2\ell+1}$ was established. In these
works, the error rate is directly linked to the smallest singular
value of $\VV_N$ and its close relative, the confluent Vandermonde
matrix \cite{gautschi_inverses_1978}.

Several other recent works by different groups investigated the
matrices $\VV_N$ under the partial clustering assumptions
\cite{akinshin_accuracy_2015,akinshin_error_2017,kunis2018,kunis2019a,li_stable_2017,li2019},
similarly showing that $\VV_N$ is only mildly ill-conditioned if
$\ell\ll s$ (see Subsection \ref{sub:related-work} below). Motivated
by the above developments, in this paper we continue the investigation
of the partial clustering model.

\subsection{Contributions}
We suppose that the nodes $\{x_j\}$ are divided into disjoint groups
(clusters), each of which is contained in an interval of length at
most $h \lesssim {1\over N}$, while the inter-cluster distances are at
least $\theta\gtrsim {1\over N}$ (see Definition
\ref{def.partial.cluster} below). Our main result (Theorem
\ref{thm.orth.spaces}) establishes that the subspaces of $\CC^{N+1}$
corresponding to each cluster (the so-called ``cluster subspaces'',
see Definition \ref{def.l.x} below) are nearly orthogonal. In more
detail, we show that for large enough $N\theta$ and small enough
$N h$, the complementary subspace angle between each pair of cluster
subspaces is at most ${c_1\over{N\theta}} + c_2N h$ for some constants
$c_1,c_2$ depending only on the multiplicities of (number of nodes in)
the clusters. {\it As a result, spectral analysis of $\VV_N$ is
  significantly simplified, reducing the problem to the analysis of
  each cluster separately} (see Theorem \ref{thm.union}). To
demonstrate this general principle, we establish the following results
for the case that the points are approximately uniformly distributed
in each cluster:
\begin{enumerate}
\item We derive full asymptotic description of \emph{all the singular
  values} of $\VV_N$ (Theorems \ref{thm:single.cl}, \ref{thm.union}).
\item In the particular case where the size of all the clusters is of the same order $h$, 
  the singular values of $N^{-1/2}\VV_N$ have the following simple scales (up to constants): 
  $$(Nh)^0,....,(Nh)^{\ell-1},$$ 
  where $\ell$ is the maximal multiplicity of any cluster. Furthermore, the number of
  singular values scaling as $(Nh)^{j-1}$ is exactly equal to the number of clusters of
  multiplicity at least $j$ (see Corollary \ref{cor:full-sing-vals}).
\item In Theorem \ref{thm:ls-accuracy} we obtain \emph{componentwise}
  stability bounds of the linear least squares problem
  $$\min_{\av} \|\VV_N(\xv) \av - \bv\|_2.$$ In particular, 
  we show that the entries of $\av$ corresponding to the nodes of $\xv$ inside 
  a cluster of size $h$ and multiplicity $\ell$ (i.e. $h,\ell$ may be different for different clusters), 
  have condition numbers proportional to $(Nh)^{1-\ell}$, with the proportionality
  constant scaling \emph{linearly} with $s$. In contrast, without
  prior geometric assumptions, all the entries have condition number
  on the scale of $\left(N\sep\right)^{1-s}$ (where $\eta$ is the global minimal separation of the nodes).
\end{enumerate}

\subsection{Related work and discussion}\label{sub:related-work}

The scaling
$\sigma_j\left({1\over{\sqrt{N}}}\VV_N\right) \approx
\left(N\sep\right)^{j-1}$ for a single cluster can also be derived
from \cite{slepian_prolate_1978,lee1992}. In the proof of Theorem
\ref{thm:single.cl} we use a particular technique based on Taylor
expansion of the kernel matrix $\VV_N^H \VV_N$, used in
\cite{wathen2015}. It will be interesting to investigate the
possibility of extending our result to more general types of matrices,
for instance those considered in \cite{lee1992}. Another interesting
question is to allow the nodes of the Vandermonde matrix to be in a
small annulus containing the unit circle, as in \cite{pan2016}.

Several previous works studied the behaviour of the minimal singular
value of clustered Vandermonde matrices in the regime $N\sep \ll 1$.
Below, positive constants that are independent of $N,\sep$ are
indicated by $c_1,c_2,\dots,c,c',\dots$. From Corollary
\ref{cor:full-sing-vals} it directly follows that
\begin{equation}\label{eq:sigma-min}
  \sigma_{\min}\left({1\over{\sqrt{N}}}\VV_N\right) \geq c (N\sep)^{\ell-1},\quad N\sep < c',
\end{equation}
where, again, $\ell$ is the largest multiplicity. This scaling has
been previously established in
\cite{batenkov2018a,li_stable_2017,kunis2019a}, by completely
different techniques and under additional conditions. In the
following, we briefly compare those results to ours.
\begin{itemize}
\item The bound \eqref{eq:sigma-min} was first established in
  \cite{batenkov2018a} in the regime $N \theta \ge c_1 $.  However, it
  was also required that the entire node set $\xv$ be contained in an
  interval of length ${1\over{s^2}}$. Compared with \cite{batenkov2018a} we similarly 
  require that $N \theta \ge c_2 $, but the node set $\xv$ is no longer restricted to such a tiny
  interval.
\item In \cite{kunis2019a} (building upon \cite{li_stable_2017}),
  \eqref{eq:sigma-min} was shown to hold with $c'=1$ but under further
  restriction of the form
  \begin{equation}\label{eq:kunis-condition}
    N\theta > c_3(\gamma) (N\sep)^{-\gamma},
  \end{equation}
  where $\gamma>0$ can be arbitrarily small. However in this case
  $\lim_{\gamma\to 0}c_3(\gamma) = \infty$ and also
  $\lim_{\gamma\to 0}c(\gamma)=0$ where $c$ is the constant in
  \eqref{eq:sigma-min}. To make a comparison, let us fix
  $\theta,\gamma$ and consider what values of $\sep$ are covered,
  first by our result: $\sep \in \left(0,c'N^{-1}\right]$ and then by
  \cite{li_stable_2017,kunis2019a}:
  $\sep \in [c''N^{-(1+\gamma^{-1})}, N^{-1}]$.  Note that:
  \begin{itemize}
  \item The regime $\sep \in \left(0,c' N^{-1}\right]$ allows $\sep\to 0$ for a
    fixed $N$;
  \item If $N$ is sufficiently large then the regimes overlap, and all
    values of $\sep\in\left(0,N^{-1}\right]$ are either covered by the
    results of this paper or those of
    \cite{li_stable_2017,kunis2019a}.
  \end{itemize}
\item Our constant $c$ in \eqref{eq:sigma-min} is not explicit,
  while the authors of \cite{kunis2019a} managed to prove that under
  the condition \eqref{eq:kunis-condition} with
  $\gamma = \frac{\ell-1}{2}$, the constant $c(\gamma)$ is of order
  $C^{-\ell}$, for an absolute constant $C$ (in \cite{batenkov2018a} a
  much worse estimate $c \sim s^{-2s}$ was given). The scaling
  $c \sim C^{-\ell}$ can be shown to be optimal (up to the magnitude
  of the absolute constant $C$), see \cite[Example 5.1]{kunis2019a}.
  Simulations suggest that \eqref{eq:sigma-min} holds with
  $c\sim C^{-\ell}$ whenever $N \theta \ge c_4$, i.e. the clusters
  separation should only be large with respect to $\frac{1}{N}$,
  regardless of the relation between $N$ and $\sep$.  We plan to close
  this gap in the constant in a future publication.
\end{itemize}

In addition, our results have consequences for the analysis of
super-resolution problem and algorithms, both on-grid and off-grid
\cite{batenkov2019a,batenkov2018a,donoho_superresolution_1992,li_stable_2017,li2019}. In
this context, it should also be interesting to investigate low-rank
approximation for the covariance matrices
\cite{beckermann2019,yang2018}.

We hope that using the cluster subspace orthogonality it will be
possible to provide an accurate description of the \emph{singular
  vectors}, in particular, their spectral concentration
properties. These questions are important in e.g. time-frequency
analysis and sampling of multiband signals \cite{hogan_duration_2011}.

\subsection{Organization of the paper}
In Section \ref{sec:main-results} we establish some notation and
formulate our main results. In Section \ref{sec:ortho} we develop the
necessary tools and prove Theorem \ref{thm.orth.spaces}. In Section
\ref{sec:single.cluster} we analyze the case of a single cluster and
prove Theorem \ref{thm:single.cl}. In Section
\ref{sec:multi-cluster-proofs} we analyze the multi-cluster setting
and prove Theorems \ref{thm.union} and \ref{thm:ls-accuracy}. In
Section \ref{sec:numerics} we present results of numerical experiments
validating our main results.
\subsection{Acknowledgements}
 The research of GG and YY is supported in part by the Minerva Foundation.

\section{Main results}\label{sec:main-results}

\subsection{Notation}

For a matrix $\AAA$, $\AAA^H$ denotes the Hermitian transpose of
$\AAA$, and $\AAA^{\dagger}$ denotes the Moore-Penrose pseudoinverse
\cite{ben-israel2003} of $\AAA$. The $k$-th component of a vector
$\xvec$ is denoted by $\left(\xvec\right)_k$, and $(i,j)$-th entry of
a matrix $\AAA$ is denoted by
$\left(\AAA\right)_{i,j}$. We denote 
the spectral, the maximum and the Frobenius norm of $\AAA$, respectively, by:
$\|\AAA\|=\max_{\|\av\|=1} \|\AAA\av\|$, $\|\AAA\|_{\max}=\max_{j,k} |(\AAA)_{j,k}|$ and
$\|\AAA\|_{F}=\left( \sum_{j,k} |(\AAA)_{j,k}|^2 \right)^\frac12$.

The following relations are standard and we use them frequently:
\begin{equation}\label{eq.matnorm}
	\|\AAA\|_{\max} \leq \|\AAA\| \leq \|\AAA\|_F \leq \sqrt{r} \|\AAA\|, \quad \text{ if $\AAA$ is of rank $r$}.
\end{equation}

For any $\AAA$ as above, we will refer to its singular values in a decreasing order 
and list them as 
$$\sigma_{\max}(\AAA)=\sigma_1(\AAA)\ge\ldots\ge \sigma_{\min(m,n)}(\AAA)=\sigma_{\min}(\AAA).$$ 

We use the Landau symbols $\MO$ for an asymptotic upper bound and $\Theta$ for asymptotically equal up to constants. 

\smallskip

Now we define the clustering configuration of the nodes.
\begin{definition}[Wrap-around distance]
	For $x,y \in \RR$, we denote the wrap-around distance 
	$$\Delta(x,y)= |\Arg \exp{\imath (x-y)}| = |x-y \mod (-\pi,\pi]| \in \left[0,\pi\right],$$ 
	where for $z\in {\mathbb C} \backslash \{0\}$, $\Arg(z)$
	is the principal value of the argument of $z$, taking values in $\left(-\pi,\pi\right]$.
\end{definition}

\begin{definition}[Single cluster configuration]\label{def.sigle.cluster}
  The node set $\xv=\left\{ x_1,\dots,x_s\right\} \subset (-\pi,\pi]$
  is said to form
  \begin{itemize}
  \item an $(h,s)$-cluster if
    $$
     \forall x,y \in \xv, x \neq y: \quad 0 < \Delta(x,y) \le h;
    $$
  \item an $(h,\tau,s)$-cluster, for some $\tau>0$, if
    $$
    \forall x,y \in \xv, x\neq y: \quad \tau h \le \Delta(x,y) \le h.
    $$
  \end{itemize}
\end{definition}

\begin{remark}
  Clearly, an $(h,\tau,s)$ cluster is in particular an
  $(h,s)$-cluster, where in addition we assume that the nodes
  are approximately uniformly distributed within the cluster.
\end{remark}

\begin{definition}[Multi-cluster configuration]\label{def.partial.cluster}
  The node set $\xv=\left\{ x_1,\dots,x_s\right\} \subset (-\pi,\pi]$
  is said to form an $((h^{(j)}, s^{(j)})_{j=1}^M, \theta)$
  (respectively,
  $((h^{(j)}, \tau^{(j)}, s^{(j)})_{j=1}^M, \theta)$)-clustered
  configuration if there exists an $M$-partition
  $\xv=\biguplus_{j=1}^M \cj$, such that for each
  $j\in\left\{1,\dots,M\right\}$ the following conditions are
  satisfied:
  \begin{itemize}
  \item $\cj$ is an $(h^{(j)}, s^{(j)})$ (respectively, an
    $(h^{(j)}, \tau^{(j)}, s^{(j)})$)-cluster;
  \item $ \Delta(x,y) \geq \theta > 0,\quad \forall x\in\cj,\;\forall y \in \xv\setminus\cj$.
  \end{itemize}
\end{definition}

Below we write $C_k(s_1,s_2,\ldots,s_j)$ or $N_k(s_1,s_2,...,s_j)$,
for some indexes $k,j$ and parameters $s_1,\ldots,s_j$, to
indicate a constant that depends only on $s_1,\ldots,s_j$.

\subsection{Cluster subspace orthogonality}

\begin{definition}[Cluster subspace]\label{def.l.x}
  Let $\xv=\left\{x_1,\dots,x_s\right\} \subset (-\pi,\pi]$ and let
  $\vv_1,\ldots,\vv_s$ denote the columns of the Vandermonde matrix
  $\VV_N(\xv)$.  We denote by $\Lx$ the subspace spanned by
  $\vv_1,\ldots,\vv_s$, i.e.
	$$\Lx:= L(\xv,N) = \Span \{ \vv_1,\ldots,\vv_s\} \subset \CC^{N+1}.$$  
\end{definition}

\begin{definition}[Minimal principal angle]\label{def.subspace.angles}
  For two subspaces $L_1,L_2 \subset \CC^{N+1}$, the minimal principal
  angle $\angle_{\min}(L_1, L_2)$ between $L_1$ and $L_2$, taking
  values in $[0,\frac{\pi}{2}]$, is defined as
  $$\angle_{\min}(L_1, L_2):=\min_{v\in L_1\backslash\{0\},u\in
    L_2\backslash\{0\}} \arccos \left(\frac{|\langle
      v,u\rangle|}{\|v\|\cdot \|u\|}\right).$$
\end{definition}

Our first main result, proved in Section \ref{sec:ortho}, reads as follows.

\begin{theorem}[Cluster subspaces orthogonality]\label{thm.orth.spaces}
  Let $\xv$ and $\yv$ form an $(h^{(1)},s_1)$- and
  $(h^{(2)},s_2)$-clusters, respectively, such that
	\begin{align*}\label{eq.cluster.sep}
          \Delta(x,y) & \ge \theta	& \forall x\in\xv, y\in\yv.	
	\end{align*}
	Put $h=\max\left(h^{(1)},h^{(2)}\right)$. Then there exist
        positive constants $\Cl{const.low}$, $\Cl{const.high}$,
        $\Cl{subsapce.angle.srf}$ and $\Cl{subsapce.angle.N}$,
        depending only on $s_1$ and $s_2$, such that for all $N$ with
        $\Cr{const.low} \le N \le \frac{\Cr{const.high}}{h}$ we have
	\begin{equation}\label{eq.thm.angle}
		\angle_{\min}(L(\xv,N),L(\yv,N))
		\ge \frac{\pi}{2} - \frac{\Cr{subsapce.angle.N}}{N\theta} - \Cr{subsapce.angle.srf}Nh.
	\end{equation}
\end{theorem}

\begin{remark}
  Note that Theorem \ref{thm.orth.spaces} holds irrespective of the
  inner structure of each cluster.
\end{remark}

\begin{remark}\label{rem:nontrivial-bound}
  Clearly, if
  $N>\max\left(\Cr{const.low},{4\over\pi}\Cr{subsapce.angle.N}\right)
  \cdot \max\left(1,\theta^{-1}\right)$ and
  $Nh<\min\left(\Cr{const.high},{\pi\over
      {4\Cr{subsapce.angle.srf}}}\right)$, then $\angle_{\min}$ is
  guaranteed to be positive. So, Theorem \ref{thm.orth.spaces} will
  always produce a nontrivial bound for sufficiently large $N$ and
  sufficiently small $Nh$.
\end{remark}

\begin{remark}
  All the constants in Theorem \ref{thm.orth.spaces} (except
  $\Cr{const.low}$) can be given explicitly. However, we feel that
  little is to be gained by doing so, as these constants are relatively
  complicated and we have not tried to optimize them. For instance,
  they depend on the smallest eigenvalue of the normalized Hilbert
  matrix, a quantity which has no known non-asymptotic closed formula
  (see Remark \ref{remark.hilbert.asymptotic}). That said, note that
  the asymptotic behavior is captured accurately, as our numerical
  experiments in Section \ref{sec:numerics} demonstrate.

\end{remark}

\subsection{Full spectral description}

Now we establish {\it accurate estimates for all the
  singular values of $\VV$}.

First, using the orthogonality result (Theorem \ref{thm.orth.spaces}),
we show in Section \ref{sec:proof.orth} that the set of all the
singular values of $\VV$ equals, up to a small multiplicative
perturbation, to {\it the union of the sets of singular values of the
  sub-matrices of $\VV$, corresponding to the clusters}.

\begin{theorem}[Multi-cluster Vandermonde matrix singular values]\label{thm.union}
  Suppose that the node set
  $\xv=\left\{x_1,\dots,x_s\right\} \subset (-\pi,\pi]$ forms an
  $((h^{(j)}, s^{(j)})_{j=1}^M, \theta)$-clustered
  configuration, and consider the Vandermonde matrix $\VV_N(\xv)$ and
  its sub-matrices formed by each cluster,
  $\VV_N(\cv^{(1)}),\ldots,\VV_N(\cv^{(M)})$. Let
  $$
  \sigma_1\ge \ldots \ge \sigma_s
  $$ 
  be the singular values of $\VV_N(\xv)$ in non-increasing order.
  Further, let
  $$
  \tilde{\sigma}_1\ge\ldots\ge \tilde{\sigma_{s}}
  $$ 
  be all the singular values of the sub-matrices
  $\{\VV_N(\cv^{(j)})\}$, also in non-increasing order.
	
  Put $h=\max_{j}(h^{(j)})$. Then there exist positive
  constants $\Cl{multi.cluster.N.theta}$, $\Cl{multi.cluster.N.h}$,
  $\Cl{vandermonde.union.N}$ and $\Cl{vandermonde.union.N.h}$,
  depending only on $s^{(1)},\dots,s^{(M)}$, such that for all $N$
  satisfying
  $\frac{\Cr{multi.cluster.N.theta}}{\theta} \le N \le
  \frac{\Cr{multi.cluster.N.h}}{h}$ we have
  \begin{align}
    \left(1-\frac{\Cr{vandermonde.union.N}}{N\theta}-\Cr{vandermonde.union.N.h}Nh\right)^{\frac{1}{2}} \tilde{\sigma}_j \le \sigma_j\le 
    \left(1+\frac{\Cr{vandermonde.union.N}}{N\theta}+\Cr{vandermonde.union.N.h}Nh\right)^{\frac{1}{2}} \tilde{\sigma}_j, && j=1,\ldots,s.
  \end{align}
\end{theorem}

\begin{remark}
  In the proof of Theorem \ref{thm.union} it is ensured that
  $1-\frac{\Cr{vandermonde.union.N}}{N\theta}-\Cr{vandermonde.union.N.h}Nh>
  0$, see Proposition \ref{prop:mult.clust.angl} and in particular
  \eqref{eq.angle.j.k.spaces}. Compare this with Remark
  \ref{rem:nontrivial-bound}.
  
  The dependency of the constants on $s$ is only linear. See Remark \ref{remark.const.maxcluster}.
\end{remark}

Thus, the analysis of the spectrum of $\VV_N$ is reduced to looking at
each cluster separately. To that effect, our next result (proved in
Section \ref{sec:single.cluster}) provides the decay rates for the
singular values corresponding to a single cluster, assuming that the
distribution of the nodes inside the cluster is approximately uniform.

\begin{theorem}[Single cluster singular values]\label{thm:single.cl}
  Let $\xv$ form an $\left(h,\tau,s\right)$-cluster. Then there exist constants
  $\Cl{single.cluster.Nh}(\tau,s)$, $\Cl{sing.lower.1}(\tau,s)$
  and $\Cl{sing.upper.1}(s)$, such that for all
  $N\geq s$ and $Nh \leq \Cr{single.cluster.Nh}$ we
  have
\begin{equation}
  \label{eq:sing.single.final}
  \Cr{sing.lower.1} N^{1\over 2} \left(Nh\right)^{j-1} \leq \sigma_j\left(\VV_N\left(\xv\right)\right) \leq \Cr{sing.upper.1} N^{1\over 2}\left(Nh\right)^{j-1},\quad j=1,\dots,s. 
\end{equation}
\end{theorem}

Combining Theorems \ref{thm.union} and \ref{thm:single.cl} provides
complete scaling of all the singular values of $\VV_N(\xv)$ for the
case of approximately uniform clusters (i.e.
$\min_j\tau^{(j)}\geq \tau > 0$).  If we assume, in addition, that all
the cluster sizes $h^{(j)}$ are of the same order (for simplicity we
may take them to be equal to each other), then we have a particularly
simple description of the spectrum of $\VV_N$ as follows.
\begin{corollary}[Entire spectrum]\label{cor:full-sing-vals}
  Let $\xv$ form an
  $((h^{(j)}, \tau^{(j)},s^{(j)})_{j=1}^M, \theta)$-clustered
  configuration, and furthermore suppose that
  $h^{(1)}=h^{(2)}=\dots=h^{(M)}=h$. For each
  $j=1,2,\dots,\max_{j} s^{(j)}$, define
  $$
  \ell_j := \#\{1\leq k \leq M: j \leq s^{(k)}\}.
  $$
  Then, there exist constants $\Cl{full.N.theta}$ and $\Cl{full.N.h}$
  such that for all $N\theta \geq \Cr{full.N.theta}$ and
  $Nh \leq \Cr{full.N.h}$ there are precisely $\ell_j$ singular values
  of $\VV_N(\xv)$ scaling like $\asymp N^{1\over
    2}\left(Nh\right)^{j-1}$. All the constants in the statement
  depend only on $\left(s^{(j)},\tau^{(j)}\right)_{j=1}^M$.
\end{corollary}

\subsection{Accuracy of least squares problems}

\begin{definition}[Cluster index set]\label{def.cluster.index}
  For a multi-cluster $\xv=\left\{x_1,\dots,x_s\right\}$, for each
  $1\leq j \leq M$, we define $C_j=C_j\left(\xv\right)$ to be the
  indices of all the nodes in cluster $j$, i.e.
  $$
  C_j\left(\xv\right) = \left\{\ell: x_{\ell} \in \cv^{(j)}\right\}.
  $$
\end{definition}

\begin{definition}[Least squares solution]\label{def.ls.solution}
  Given a multi-cluster $\xv$ as in Definition
  \ref{def.partial.cluster}, $N > s$, and a vector $\bv\in\CC^{N+1}$,
  define
  $$
  \av\left(\xv,\bv\right):=\arg\min_{\av}\|\VV_N(\xv)\av - \bv\|_2.
  $$
\end{definition}

In Section \ref{sec:proof.ls} we prove the following result.

\begin{theorem}\label{thm:ls-accuracy}
  Suppose that the node set
  $\xv=\left\{x_1,\dots,x_s\right\} \subset (-\pi,\pi]$ forms an
  $((h^{(j)}, \tau^{(j)},s^{(j)})_{j=1}^M,\allowbreak \theta)$-clustered
  configuration, with $h=\max_{1\leq j\leq M}(h^{(j)})$. Let $N$
  satisfy
  $\frac{\Cl{ls.N.theta}}{\theta} \leq N \leq \frac{\Cl{ls.N.h}}{h}$
  for certain constants $\Cr{ls.N.theta}(s^{(1)},\dots,s^{(M)})$ and
  $\Cr{ls.N.h}(s^{(1)},\dots,s^{(M)})$ to be specified in the
  proof. Next, let $\av_0\in \CC^s$ be arbitrary and
  $\bv_0 = \VV_N(\xv)\av_0$. Then for each
  $j\in\left\{1,2,\dots,M\right\}$ there exists a constant
  $\Cl{ls.ub}(s^{(j)},\tau^{(j)})$ such that for all
  $\ell \in C_j\left(\xv\right)$ we have
  \begin{equation}
    \label{eq:ls-componentwise-bound}
    \biggl|\bigl(\av\left(\xv,\bv\right)-\av_0\bigr)_{\ell}\biggr| \leq \Cr{ls.ub} s \left(\frac{1}{Nh^{(j)}}\right)^{s^{(j)}-1} \|\bv-\bv_0\|_\infty.
  \end{equation}
\end{theorem}

The above result shows that for multi-cluster distributions of the
nodes, the componentwise condition numbers are much more accurate than
the standard condition number. Indeed, without any geometric
assumptions on the node distribution, the classical condition number
$\kappa(\VV_N) = \frac{\sigma_{\max}(\VV_N)}{\sigma_{\min}(\VV_N)}$ is
well-known to grow exponentially with $s$, see
e.g. \cite{beckermann2019,pan2016} and references therein. In
contrast, the errors in \eqref{eq:ls-componentwise-bound} grow
exponentially with the multiplicities of each cluster, and only
linearly in the overall number of nodes (compare also with
\cite[Corollary 3.10]{batenkov2018a}).

\begin{remark}
  If one considers perturbations in $\xv$ as well, the stability
  analysis becomes more complicated, see
  e.g. \cite{bjorck1991,higham_survey_1994}. The main point we would
  like to emphasize here is that the components of $\av$ have
  different condition numbers according to the multiplicity of the
  nodes of $\xv$ in the corresponding cluster.
\end{remark}

\section{Orthogonality of cluster subspaces}\label{sec:ortho}

In this section we prove Theorem \ref{thm.orth.spaces}. To facilitate
the reading, let us start with a short overview of the steps.
\begin{enumerate}
\item In Section \ref{sec:limit-spaces-bases} we introduce several
  objects associated with an $(h,s)$-cluster $\xv$: a particular basis
  for $L=L(\xv,N)$, called the ``divided difference basis''; the limit
  space $\bar{L}$; and a certain basis for $\bar{L}$ called the
  ``limit basis''. We show that the spaces $L$ and $\bar{L}$ are
  ``close'', in the sense that the divided difference basis vectors of
  $L$ are close to the corresponding limit basis vectors of $\bar{L}$,
  up to order $\MO(Nh)$.
\item Next, in Subsection \ref{sec.limit.basis.non.deg} we show that
  limit basis is well-conditioned (i.e. the smallest singular value of
  the corresponding matrix is effectively bounded from below for large
  enough $N$).
\item Given two different clusters of the nodes, $\xv$ with $s_1$
  nodes and $\yv$ with $s_2$ nodes, in Subsection
  \ref{sec.orth.limit.spaces} we prove the following key property: the
  limit spaces $\bar{L}(\xv,N)$, $\bar{L}(\yv,N)$, are nearly
  orthogonal, with $\angle_{\min}(\bar{L}_1,\bar{L}_2)$ being of order
  $\frac{\pi}{2}-\MO\left(\frac{1}{N}\right)$.
\item Combining the above results, in Subsection
  \ref{sec.orth.cluster.spaces} we conclude that the angle between any
  $\vv\in L_1$ and $\uv\in L_2$ is at least
  ${\pi\over 2}-\MO(Nh)-\MO({1\over N})$, completing the proof.
\end{enumerate}

\subsection{Spaces and bases}\label{sec:limit-spaces-bases}
Let $\xv$ form an $(h,s)$-cluster as in Definition
\ref{def.sigle.cluster}, fix $N\geq s-1$, and, as per Definition
\ref{def.l.x}, let
$$
\Lx = L(\xv,N)= \Span\left\{\vv_1,\ldots,\vv_s\right\}\subset \CC^{N+1},\quad \VV_N(\xv) = \left[\vv_1,\ldots,\vv_s\right].
$$

We start by constructing a basis for $\Lx$, different from
$\left\{\vv_1,\dots,\vv_s\right\}$, which is given by certain divided
differences of the vectors $\vv_1,\ldots,\vv_s$.  For the reader's
convenience we recall the definition of divided differences below, and
list several of their propoerties in Lemma
\ref{lem:dd-properties} in Appendix \ref{appen.dd}. For
further references see e.g. \cite{de2005divided} and
\cite[Section 6.2]{batenkov_geometry_2014}.

\begin{definition}[Divided finite differences]\label{def.divided.diff}
  Let an arbitrary sequence of points $(t_1,t_2,\dots,)$ be given
  (repetitions are allowed). For any $n=1,2,\dots,$, and for any
  smooth enough real-valued function $f$, defined at least on
  $t_1,\dots,t_n$, the $n-1$-st divided difference $[t_1,\dots,t_n]f$
  is the $n$-th coefficient, in the Newton form, of the (uniquely
  defined) Hermite interpolation polynomial $p$, which agrees with $f$
  and its derivatives of appropriate order on $t_1,\dots,t_n$,
  i.e. $[t_1,\dots,t_n]f\equiv [t_1,\dots,t_n]p$ where
  \begin{align*}
  f^{(\ell)}(t_j) &= p^{(\ell)}(t_j): \quad 1 \leq j \leq n, \;0 \leq \ell < d_j:=\#\left\{i:\;\;t_i=t_j\right\}; \\
    p(t) &\equiv \sum_{j=1}^n \left\{[t_1,\dots,t_j]p\right\}\prod_{k=1}^{j-1} (t-t_j).
  \end{align*}
  Divided differences of complex-valued functions are defined by
  applying them to the real and the imaginary parts separately:
  $$
  [t_1,\dots,t_n](f+\imath g) = [t_1,\dots,t_n]f + \imath [t_1,\dots,t_n]g.
  $$

  For a vector of functions $\fv = \left(f_1,\dots,f_m\right)$, we
  denote
  $$
  [t_1,\dots,t_n]\fv := \left([t_1,\dots,t_n]f_1,\dots,[t_1,\dots,t_n]f_m\right)\in\CC^m.
  $$
\end{definition}

Denote by
  $$
  \vv_N(x)=\left( 1, e^{\imath x}, e^{2\imath x}, \dots, e^{N\imath x}\right)
  $$
  the vector of $N+1$ exponential functions. Note that by
  \eqref{eq:dd-recursive-comp} the standard
  basis for $L(\xv,N)$ can be simply written as
  \begin{equation}\label{eq:vv-as-dd}
    \vv_j \equiv [x_j]\vv_N(x),\quad j=1,\dots,s.
  \end{equation}

  \begin{definition}\label{def.divided.diff.basis}
    Given $\xv=\left\{ x_1,\dots,x_s\right\} \subset (-\pi,\pi]$ and a
    positive integer $N$, define, for every
    $j\in\left\{1,\dots,s\right\}$ the following vector:
    \begin{align}
      \label{eq:dd-basis-def}
      \wv_j &:= (j-1)![x_1,\dots,x_j]\vv_N(x), & \tilde{\wv}_j:=\frac{\wv_j}{\|\wv_j\|}.
    \end{align}
    The ordered set $\WB=\WB(\xv)=\left\{\wv_1,\dots,\wv_s\right\}$
    is called the \emph{divided difference basis} to $L(\xv)$, and the
    ordered set
    $\widetilde{\WB}=\widetilde{\WB}(\xv) =
    \left\{\tilde{\wv}_1,\dots,\tilde{\wv}_s\right\}$ is called the
    normalized divided difference basis to $L(\xv)$.
  \end{definition}

  \begin{definition}[Limit basis and limit space]\label{def:limit.vectors}
    Given $\zeta \in (-\pi,\pi]$ and positive integers $N,s$, define,
    for every $j\in\left\{1,\dots,s\right\}$ the following vector:
    \begin{align}  
      \label{eq:lmit-basis-def}
      \uv_j &:= (j-1)![\underbrace{\zeta,\dots,\zeta}_{j\text{ times}}]\vv_N(x), & \tilde{\uv}_j = \frac{\uv_j}{\|\uv_j\|}.
    \end{align}
    The ordered set
    $\UB(\zeta,N,s):=\left\{\uv_1,\allowbreak \ldots,\uv_s \right\}
    \subset \CC^{N+1}$ is called the $(\zeta,N,s)$ limit basis, and
    the ordered set
    $\widetilde{\UB}=\widetilde{\UB}(\zeta,N,s)=\left\{\tilde{\uv}_1,\ldots,\tilde{\uv}_s\right\}$
    is called the normalized $(\zeta,N,s)$ limit basis.

    We further define the limit space at $\zeta$ as
    $$
    \bar{L}(\zeta,N,s) := \Span
    \left\{\UB\left(\zeta,N,s\right)\right\}\subset \CC^{N+1}.
    $$
  \end{definition}
  \begin{definition}[Cluster limit space]\label{def:cluster-lim-space}
    For $\xv$ forming an $(h,s)$-cluster we define the cluster limit
    space $\bar{L}(\xv,N)$ by
    $$
    \bar{L}(\xv,N) := \bar{L}(x_1,N,s) \subset \CC^{N+1}.
    $$      
  \end{definition}

  \begin{remark}
    The choice of the point $x_1$ in the definition of $\bar{L}(\xv,N)$
    is arbitrary, in the sense that all subsequent results hold true
    if we replace $x_1$ in Definition \ref{def:cluster-lim-space} with
    another node $x_j\in\xv$, or, even more generally, with any other
    point $x\in[\min_j x_j,\max_j x_j]$. Furthermore, note that
    $\bar{L}(\xv,N)$ depends neither on the cluster size $h$ nor on
    the relative positions of the points inside the cluster.
  \end{remark}

  Below we establish several properties of the sets $\WB$,
  $\widetilde{\WB}$, $\UB$, $\widetilde{\UB}$ and the relationships
  between them, which will be used in the rest of the section, towards
  the proof of Theorem \ref{thm.orth.spaces} in Subsection
  \ref{sec.orth.cluster.spaces}.
  
  \begin{proposition}\label{prop:basis-properties}
    Let $\xv$ form an $(h,s)$-cluster, and let $N\geq s-1$.
    \begin{enumerate}
    \item $\Span\left\{\WB(\xv)\right\}=L(\xv)$.
    \item Each $\uv_j\in\UB(\zeta,N,s)$ is explicitly given by
      \begin{align}\label{eq:ujk-explicit}
        \uv_{j,k}&=\frac{\dt^{j-1}}{\dt x^{j-1}}
        e^{\imath k x}\Big|_{\zeta}=(\imath k)^{j-1}e^{\imath k\zeta},&
        j=1,\ldots,s,\; k=0,\ldots,N,
      \end{align}
    \item $\UB(\zeta,N,s)$ is a linearly independent set, i.e. it is
      indeed a basis for $\bar{L}(\zeta,N,s)$.
    \item Putting $\zeta=x_1$, then
      \begin{align*}
	\lim_{h\to 0} \wv_j &= \uv_j \in \UB\left(x_1,N,s\right),& j\in\left\{1,\ldots,s\right\}.
      \end{align*}
    \item With
      $\Cl{const.norm.l.b}=\Cr{const.norm.l.b}(s):=\frac{1}{\sqrt{2s-1}}$,
      we have
	\begin{alignat}{2}\label{eq.prop.norm.u}
          \Cr{const.norm.l.b} N^{j-\frac{1}{2}} \le & \|\uv_j\| \le
          N^{j-\frac{1}{2}},\quad \uv_j \in \UB(\zeta,N,s), \;
          j\in\left\{1,\dots,s\right\}.
        \end{alignat}
      \item With $\Cl{cnst.divided.derivative.diff}:=2\sqrt{2}$, we have
        \begin{equation}
          \label{eq:diff.before.limit}
          \left\|\tilde{\uv}_{j}-\tilde{\wv}_{j}\right\| \le
          \Cr{cnst.divided.derivative.diff}N h,\quad \tilde{\uv}_j \in \widetilde{\UB}(x_1,N,s),\; \tilde{\wv}_j \in \widetilde{\WB}(\xv),\;j\in\left\{1,\dots,s\right\}.
        \end{equation}
      \end{enumerate}
  \end{proposition}

  \begin{proof}

    In the proofs below, we use results from Appendices \ref{appen.dd} and
    \ref{appen.power.sums}.
    
    \begin{enumerate}
    \item Using the extended form \eqref{eq:dd-distinct-explicit}
      together with \eqref{eq:vv-as-dd}, we see that each vector
      $\wv_j$ is a linear combination of the original basis vectors
      $\vv_1,\ldots,\vv_j$, i.e.
      $$\wv_j=\sum_{k=1}^j \frac{1}{\prod_{k\ne j}(x_j-x_k)}\vv_k.$$
      Hence the set $\wv_1,\ldots,\wv_s$ is given by a triangular
      transformation, with non-zero coefficients, of the original
      basis $\vv_1,\ldots,\vv_s$, and therefore forms another basis to
      the subspace $L(\xv)$.
    \item Follows from \eqref{eq:dd-repeated-all} and \eqref{eq:lmit-basis-def}.
    \item Using the previous explicit formula for $\uv_j$, the matrix
      $\left[\uv_1,\dots,\uv_s\right]\in\CC^{(N+1)\times s}$ is, up to
      a diagonal factor, the Pascal-Vandermonde matrix and is known to
      be full rank (see e.g. \cite[Section 4.1]{batenkov_stability_2016}).
    \item Directly follows from the definitions and the continuity
      property \eqref{eq:dd-continuity}.
    \item For each $j\in\{1,\ldots,s\}$ we have that
      $\|\uv_j\|=\sqrt{\sum_{k=0}^N k^{2(j-1)}}$.
	Using \eqref{eq.integer.power.sum} we have 
	\begin{equation*}
	\frac{N^{j-\frac{1}{2}}}{\sqrt{2s-1}} \le \frac{N^{j-\frac{1}{2}}}{\sqrt{2j-1}} \le \|\uv_j\|\le
	N^{j-\frac{1}{2}},
	\end{equation*}
	which then proves \eqref{eq.prop.norm.u}. 
      \item First we have that
	\begin{align}\label{eq.normelized.vector.diff}
	\begin{split}
	\left\|\tilde{\uv}_{j}-\tilde{\wv}_{j}\right\| =
	\left\|\frac{\uv_j}{\|\uv_j\|}-\frac{\wv_j}{\|\wv_j\|}\right\|&=
	\frac{1}{\|\uv_j\|}\left\|\uv_j-\wv_j +\left(1-\frac{\|\uv_j\|}{\|\wv_j\|}\right)\wv_j\right\|\\
	& \le \frac{1}{\|\uv_j\|}\left(\|\uv_j-\wv_j\|+\big|  \|\uv_j\| - \|\wv_j\| \big|\right)\\
	& \le 2 \frac{\|\uv_j-\wv_j\|}{\|\uv_j\|}.
	\end{split}
	\end{align}
	Let $k\in\left\{0,1,\dots,N\right\}$. By assumption,
        $|x_{\ell}-x_m| \leq h$ for all
        $\ell,m \in \left\{1,\dots,j\right\}$. Using
        \eqref{eq:ujk-explicit}, \eqref{eq:dd-basis-def}, and applying
        \eqref{eq:dd-mean-value} to the real and the imaginary parts
        of $e^{ikx}$, we obtain
        $\xi_1,\xi_2\in [\min_{\ell} x_{\ell}, \max_{\ell} x_{\ell}]$ such that
  \begin{align*}
    \left|\uv_{j,k}-\wv_{j,k}\right|^2 &=
                                         \left|\frac{\dt^{j-1}}{\dt x^{j-1}} e^{\imath k
                                         x}\Big|_{x_1}-(j-1)![x_1,\ldots,x_j]e^{\imath kx}\right|^2 \\
                                       &= \Bigg|\frac{\dt ^{j-1}}{\dt x^{j-1}}
                                         \cos(kx)\Big|_{x_1}-\frac{\dt^{j-1}}{\dt x^{j-1}}
                                         \cos(kx)\Big|_{\xi_1}\Bigg|^2+\Bigg|\frac{\dt^{j-1}}{\dt x^{j-1}}
                                         \sin(kx)\Big|_{x_1}-\frac{\dt^{j-1}}{\dt x^{j-1}}
                                         \sin(kx)\Big|_{\xi_2}\Bigg|^2 \\
    & \leq 2k^{2j}h^2,
  \end{align*}
  where the last line is obtained by applying the standard mean value
  theorem to $\cos^{(j-1)}(kx)$ and $\sin^{(j-1)}(kx)$, respectively.

  Now $\left|\uv_{j,k}-\wv_{j,k}\right|\le \sqrt{2}k^{j}h$, which
  implies that
  \[\|\uv_j-\wv_j\|\le \sqrt{2}h\left(\sum_{k=1}^n
      k^{2j}\right)^{1/2}.\]

  Using $\|\uv_j\|=\sqrt{\sum_{k=0}^N k^{2(j-1)}}$ and
  \eqref{eq.normelized.vector.diff}, we obtain that
  \[
    \left\|\tilde{\uv}_{j}-\tilde{\wv}_{j}\right\|\leq 2 \sqrt{2}h
    \left(\frac{\sum_{k=1}^N k^{2j}}{\sum_{k=1}^N
        k^{2(j-1)}}\right)^{1/2}\leq 2\sqrt{2} hN. \qedhere
  \]
\end{enumerate}
\end{proof}

\subsection{Conditioning of the limit  basis}\label{sec.limit.basis.non.deg}

Given $\zeta \in (-\pi,\pi]$, and $N\geq s-1$, consider the normalized
limit basis $\widetilde{\UB}=\widetilde{\UB}(\zeta,N,s)$ to the limit
space $\bar{L}(\zeta,N,s) \subset \CC^{N+1}$. While for any fixed $N$
we have seen that $\widetilde{\UB}$ is linearly independent, in this
section we will furthermore establish that for sufficiently large $N$
the corresponding condition number is bounded from below by a constant
which does not depend on $N$.

	For each $\zeta$ as above, let $\UU(\zeta,N,s)\in
\CC^{(N+1)\times s}$ denote the matrix with columns
$\left\{\tilde{\uv}_1,\ldots,\tilde{\uv}_s\right\}$:
        $$
        \UU(\zeta,N,s)=[\tilde{\uv}_1,\ldots,\tilde{\uv}_s]=
	  \begin{bmatrix} 1 & 0 & \dots & 0 \\ e^{\imath\zeta} &
(\imath)e^{\imath\zeta} & \dots & (\imath)^{s-1}e^{\imath \zeta} \\
e^{\imath 2\zeta} & (\imath 2)e^{\imath 2\zeta} & \dots & (\imath
2)^{s-1}e^{\imath 2\zeta} \\ \vdots & \vdots & \vdots & \vdots \\
e^{\imath N\zeta} & (\imath N)e^{\imath N\zeta} & \dots & (\imath
N)^{s-1}e^{\imath N\zeta}
\end{bmatrix} \cdot \diag(\|\uv_1\|^{-1} ,\ldots,\|\uv_s\|^{-1}).
  	$$ 
		
	\begin{proposition}\label{prop.sing.limit}
          Given a positive integer $s$ and $\zeta \in (-\pi,\pi]$,
          there exist a monotonically increasing constant
          $\Cl[N]{const.limit.matrix.N}=\Cr{const.limit.matrix.N}(s)$
          and a monotonically decreasing constant $\Xi=\Xi(s) > 0$, such
          that for any $N\ge \Cr{const.limit.matrix.N}$,
		$$\sigma_{\min}\left(\UU(\zeta,N,s)\right) \ge \Xi.$$
		Moreover,
                $\Xi=\sqrt{\frac{\lambda_{\min}(\bar{\mathbf{H}}_s)}{2}}$,
                where $\bar{\mathbf{H}}_s$ is the normalized
                $s \times s$ Hilbert matrix defined in
                \eqref{eq:hilbert-matrix-def} below.
	\end{proposition}  
	\begin{proof}
		First we extract the $\imath^{j-1}$ factor from each column by putting 
		$$\tilde{\UU}_N = \UU(\zeta,N,s) \cdot \diag(\imath^{-0},\ldots,\imath^{-(s-1)}),$$ 
		and clearly $\sigma_{\min}\left(\tilde{\UU}_N\right)=\sigma_{\min}\left(\UU(\zeta,N,s)\right)$. 
		We therefore continue with $\tilde{\UU}_N$.
		
		\smallskip
		
		Next we consider the Gramian matrix $\tilde{\UU}_N^H \tilde{\UU}_N$, and we have that
		\begin{equation}\label{eq.gram.U'U}
			\left[\tilde{\UU}_N^H \tilde{\UU}_N\right]_{j,l} = \frac{\sum_{k=0}^{N}k^{j+l-2}}{\|\uv_j\|\cdot\|\uv_l\|}=
			\frac{\sum_{k=0}^{N}k^{j+l-2}}{\sqrt{\left(\sum_{k=0}^N k^{2(j-1)}\right)\left(\sum_{k=0}^N k^{2(l-1)}\right)}}.			
		\end{equation}

		By combining \eqref{eq.gram.U'U} and
                \eqref{eq.Faulhaber.formula} we get\footnote{Notice
                  also that an explicit bound for the $\MO(N)$ terms
                  can be obtained from Faulhaber's formula given in
                  \eqref{eq.Faulhaber.formula}.}
		\begin{equation}\label{eq.limit.hilbert.matrix}
		  \left[\tilde{\UU}_N^H \tilde{\UU}_N\right]_{j,l}  = 
		  \frac{\frac{N^{j+l-1}}{j+l-1} + \MO(N^{j+l-2})}{\sqrt{\left(\frac{N^{2j-1}}{2j-1} +
		  \MO(N^{2j-2})\right)\left(\frac{N^{2l-1}}{2l-1} + \MO(N^{2l-2})\right)}}=\frac{\sqrt{2j-1}\sqrt{2l-1}}{j+l-1} +
		  \MO\left(\frac{1}{N}\right),		
		\end{equation}
		where the $\MO$ asymptotic notation here and
                throughout the rest of the proof will always refer to
                $N$.
		
		Define the inner product $\langle\cdot,\cdot\rangle_H$ and the corresponding norm $\|\cdot\|_H$, over the vector space of
		polynomials of degree smaller than $s$, as
		\begin{equation}\label{eq.hilbert.inner.product}
			\langle P,Q\rangle_H=\int_{0}^1 P(x) \overline{Q(x)}\dt x.		
		\end{equation}
		 
		Then by \eqref{eq.limit.hilbert.matrix} we can write $\tilde{\UU}_N^H \tilde{\UU}_N$ as follows
		\begin{equation}\label{eq.limit.hilbert.plus.error}
			\tilde{\UU}_N^H \tilde{\UU}_N = \bar{\mathbf{H}}_s + \mathbf{E}_s,		
		\end{equation}
		where $\mathbf{E}_s,\bar{\mathbf{H}}_s$ are Hermitian matrices,   
		the entries of $\mathbf{E}_s$ are $\MO\left(\frac{1}{N}\right)$, and $\bar{\mathbf{H}}_s$ is 
		the normalized Hilbert matrix (see e.g \cite{choi1983tricks,todd1954condition}),
		\begin{equation}\label{eq:hilbert-matrix-def}
                  [\bar{\mathbf{H}}_s]_{j,l} = \left\langle \frac{x^{j-1}}{\|x^{j-1}\|_H},\frac{x^{l-1}}{\|x^{l-1}\|_H}\right\rangle_H=\int_{0}^1
                  \frac{x^{j-1}}{\|x^{j-1}\|_H}\frac{x^{l-1}}{\|x^{l-1}\|_H}\dt x.
                \end{equation}
		
		$\bar{\mathbf{H}}_s$ is the Gramian matrix with respect to the inner products of the form
		\eqref{eq.hilbert.inner.product}, of the normalized monomial basis
		$\frac{1}{\|1\|_H},\ldots,\frac{x^{s-1}}{\|x^{s-1}\|_H}$. Therefore, it
		is non-degenerate, and its smallest eigenvalue  $\lambda_{\min}(\bar{\mathbf{H}}_s)$ 
		is bounded from below by a positive constant depending only on $s$. 
		
		\smallskip
		
		On the other hand, since the entries of $\mathbf{E}_s$
                are $\MO\left(\frac{1}{N}\right)$ and $\mathbf{E}_s$
                is Hermitian (but not necessarily PSD), we have that
		\begin{equation}\label{eq.lambda.min.E}
			\lambda_{\min}(\mathbf{E}_s)\ge -\|\mathbf{E}_s\|\geq -\|\mathbf{E}_s\|_F \geq s\MO\left(\frac{1}{N}\right).
		\end{equation}
		
		Using \eqref{eq.lambda.min.E} we set $\Cr{const.limit.matrix.N}(s)$ to be such that for all 
		$N\ge \Cr{const.limit.matrix.N}$,
		$\lambda_{\min}(\mathbf{E}_s)\ge -\frac{\lambda_{\min}(\bar{\mathbf{H}}_s)}{2}$. Furthermore, we increase $\Cr{const.limit.matrix.N}(s)$ to be as least as large as $\Cr{const.limit.matrix.N}(1),\dots, \Cr{const.limit.matrix.N}(s-1)$.
		Then using \eqref{eq.limit.hilbert.plus.error} and Weyl's perturbation inequality, 
		for all $N \ge \Cr{const.limit.matrix.N}$,	  
		$$
			\sigma_{\min}(\tilde{\UU}_N) = \sqrt{\lambda_{\min}\left(\tilde{\UU}_N^H \tilde{\UU}_N\right)}
			\ge \sqrt{\lambda_{\min}\left(\bar{\mathbf{H}}_s\right)+\lambda_{\min}\left(\mathbf{E}\right)}
			\ge \sqrt{\frac{\lambda_{\min}\left(\bar{\mathbf{H}}_s\right)}{2}}.
		$$
	The last claim that $\Xi(s)$ is decreasing follows as deleting the last row and column of  $\HilbM_{s+1}$ gives $\HilbM_s$. Thus, $\lambda_{\min}(\HilbM_{s+1})\leq \lambda_{\min}(\HilbM_s)$ by the minimax principle.
	\end{proof}
	
	\begin{remark} \label{remark.hilbert.asymptotic}
		Clearly, $\sigma_{\min}\left(\UU(\zeta,N,s)\right)$ does not depend on $\zeta$ and we just proved that 
		\[ \sigma_{\min}\left(\UU(\zeta,N,s)\right)
                  \rightarrow \sqrt{\lambda_{\min}(\HilbM_s)}, \quad
                  N\rightarrow \infty. \] Since $\UU(\zeta,N,s)$ is
                injective whenever $N\geq s-1$ (recall Proposition
                \ref{prop:basis-properties}), we could replace
                $\Cr{const.limit.matrix.N}(s)$ by $s-1$.  Then,
                however, we would lose control over $\Xi(s)$. Here
                we can give an asymptotic lower bound for $\Xi(s)$ as
                follows. The asymptotic behavior of
                $\lambda_{\min}(\mathbf{H}_s)$ where $\mathbf{H}_s$ is
                the unnormalized Hilbert matrix is known to be
		\[ \lambda_{\min}(\mathbf{H}_s) \in \Theta\left(\sqrt{s}(1+\sqrt{2})^{-4s} \right),\quad s\rightarrow \infty,  \]
		see \cite{wilf2012finite}, equation (3.35). By normalization we lose at most another factor of $s$, ending up with a lower bound of order $(1+\sqrt{2})^{-2s}/\sqrt{s}$ for $\Xi(s)$.
	\end{remark}

\subsection{Near-orthogonality of the limit spaces}\label{sec.orth.limit.spaces}

\begin{proposition}\label{prop.norm1}
  For any two distinct points $\zeta_1,\zeta_2\in (-\pi,\pi]$ and
  positive integers $s_1,s_2$, there exists a positive constant
  $\Cl{prod.limit.final}(s_1,s_2)$, such
  that for any $N \ge \max (s_1,s_2)-1$:
  \begin{align}\label{eq:limit-inner-prod}
    \left|\langle \zv_1, \zv_2 \rangle \right| & \leq \frac{\Cr{prod.limit.final}}{\Delta\left(\zeta_1,\zeta_2\right)N}; \quad \zv_{\ell} \in \widetilde{\UB}\left(\zeta_{\ell},N,s_{\ell}\right),\;\;\ell=1,2.
  \end{align}
\end{proposition}

\begin{proof}

Consider the following limit vectors (as they appear in Definition
\ref{def:limit.vectors}):
\begin{align*}
  \UB\left(\zeta_{\ell},N,s_{\ell}\right) &= \left\{\uv_j^{(\ell)}\right\}_{j=1}^{s_{\ell}}, & \ell=1,2; \\
  \widetilde{\UB}\left(\zeta_{\ell},N,s_{\ell}\right) &= \left\{\tilde{\uv}_j^{(\ell)}\right\}_{j=1}^{s_{\ell}}, & \ell=1,2.
\end{align*}

Consider an arbitrary any pair of normalized limit vectors
$\zv_1=\tilde{\uv}^{(1)}_p$ and
$\zv_2=\tilde{\uv}^{(2)}_q$, $p=1,\ldots,s_1$,
$q=1,\ldots,s_2$. By \eqref{eq.prop.norm.u} we have:
    \begin{align}\label{eq.prop.norm.u.p.q}
      \begin{split}
        \frac{1}{\sqrt{2s_1-1}}N^{p-\frac{1}{2}} \le & \|\uv^{(1)}_p\| \le  N^{p-\frac{1}{2}},\\
        \frac{1}{\sqrt{2s_2-1}}N^{q-\frac{1}{2}} \le & \|\uv^{(2)}_q\|
        \le N^{q-\frac{1}{2}}.
      \end{split}
    \end{align}

      By \eqref{eq:ujk-explicit} we have
$$
	\langle\uv^{(1)}_p,\uv^{(2)}_q\rangle = \imath^{p+q-2}(-1)^{q-1}  \sum_{k=0}^N k^{p+q-2}z^k,
$$
with $z=e^{\imath(\zeta_1-\zeta_2)}$.  Notice that since
$\zeta_1 \ne \zeta_2$, we have $z \ne 1$. Now using Lemma
\ref{lem:norm2} we get that
\begin{equation}\label{eq.innet.prod.bound}
		\left|\langle \uv^{(1)}_p,\uv^{(2)}_q\rangle \right| = 	\left|\sum_{k=0}^N k^{p+q-2}z^k\right|\le \frac{2}{|1-z|}N^{p+q-2}.
\end{equation}
Finally we have that 
$$|1-z| = |1-e^{\imath(\zeta_1-\zeta_2)}| \ge \frac{2}{\pi}\Delta(\zeta_1,\zeta_2),$$
which together with \eqref{eq.innet.prod.bound} gives that 
\begin{equation}\label{eq:norm1}
\left|\langle \uv^{(1)}_p,\uv^{(2)}_q\rangle \right|\le \frac{\pi}{\Delta(\zeta_1,\zeta_2)}N^{p+q-2}.
\end{equation}

Combining \eqref{eq.prop.norm.u.p.q} with \eqref{eq:norm1}
we obtain:
\begin{equation*}
	\left| \langle \zv_1, \zv_2 \rangle \right| = \frac{|\langle \uv^{(1)}_p,\uv^{(2)}_q\rangle|}{\|\uv^{(1)}_p\| \|\uv^{(2)}_q\|}\le 
	\frac{\sqrt{(2s_1-1)(2s_2-1)}\pi N^{p+q-2}}{
	N^{p-\frac{1}{2}}N^{q-\frac{1}{2}}\Delta(\zeta_1,\zeta_2)},
    \end{equation*}
    finishing the proof with $\Cr{prod.limit.final}(s_1,s_2)=\pi\sqrt{(2s_1-1)(2s_2-1)}$.
\end{proof}
	
\subsection{Proof of Theorem
  \ref{thm.orth.spaces}}\label{sec.orth.cluster.spaces}

		Let $\xv, s_1, h^{(1)},\yv, s_2, h^{(2)}, \theta, h$ be as stated in Theorem \ref{thm.orth.spaces}. 
		Let $N$ be such that $\Cr{const.low} \le N \le \frac{\Cr{const.high}}{h}$, where
		$\Cr{const.low}=\Cr{const.low}(s_1,s_2),\;\Cr{const.high}=\Cr{const.high}(s_1,s_2)$,
		will be specified within the proof.
		
		\smallskip 
		
		Let $\vv \in L_1 = L(\xv,N)$ and
		$\uv \in L_2 = L(\yv,N)$ be two unit vectors. We will show that
		\begin{equation}\label{eq.oth.main}
			|\langle\vv,\uv\rangle | \le \frac{\Cl{inner.prod.clusters.N}(s_1,s_2)}{N \theta} + 
	    	\Cl{inner.prod.clusters.srf}(s_1,s_2)Nh,		
		\end{equation} 
		where
                $\Cr{inner.prod.clusters.N}=\Cr{inner.prod.clusters.N}(s_1,s_2)$
                and
                $\Cr{inner.prod.clusters.srf}=\Cr{inner.prod.clusters.srf}(s_1,s_2)$
                will be specified during the proof. Now, \eqref{eq.oth.main}
                implies that
	    $$\angle_{\min}(L_1,L_2)\ge \frac{\pi}{2} - \frac{\pi}{2}\frac{\Cr{inner.prod.clusters.N}(s_1,s_2)}{N \theta} -
	    	\frac{\pi}{2}\Cr{inner.prod.clusters.srf}(s_1,s_2)Nh,
	    $$
	    thus proving \eqref{eq.thm.angle} with $\Cr{subsapce.angle.N} = \frac{\pi}{2}\Cr{inner.prod.clusters.N}$ and 
	    $\Cr{subsapce.angle.srf} = \frac{\pi}{2}\Cr{inner.prod.clusters.srf}$. 
		
            \smallskip Recalling Definitions
            \ref{def.divided.diff.basis} and \ref{def:limit.vectors},
            let
            \begin{align*}
              \widetilde{\UB}(x_1,N,s_1)&:=\left\{\tilde{\uv}^{(1)}_1,\ldots,\tilde{\uv}^{(1)}_{s_1}\right\},
              & \UU_1&:=[\tilde{\uv}^{(1)}_1,\ldots,\tilde{\uv}^{(1)}_{s_1}]; \\
              \widetilde{\UB}(y_1,N,s_2)&:=\left\{\tilde{\uv}^{(2)}_1,\ldots,\tilde{\uv}^{(2)}_{s_2}\right\},
              & \UU_2&:=[\tilde{\uv}^{(2)}_1,\ldots,\tilde{\uv}^{(2)}_{s_2}]; \\
              \widetilde{\WB}(\xv)&:=\left\{\tilde{\wv}^{(1)}_1,\ldots,\tilde{\wv}^{(1)}_{s_1}\right\},
              & \WW_1&:=[\tilde{\wv}^{(1)}_1,\ldots,\tilde{\wv}^{(1)}_{s_1}]; \\
              \widetilde{\WB}(\yv)&:=\left\{\tilde{\wv}^{(2)}_1,\ldots,\tilde{\wv}^{(2)}_{s_2}\right\},
              & \WW_2&:=[\tilde{\wv}^{(2)}_1,\ldots,\tilde{\wv}^{(2)}_{s_2}].
            \end{align*}
            Furthermore, put
            \begin{equation}\label{eq.W1.W2.as.pretubation.of.U1.U2}
              \WW_1 = \UU_1 + \EE_1,\tab \WW_2 = \UU_2 + \EE_2.
            \end{equation}
            Then by \eqref{eq:diff.before.limit}
            and equation \eqref{eq.matnorm},
            \begin{equation}\label{eq.norm.error}
              \|\EE_j\| \le
              \|\EE_j\|_F\leq \sqrt{s} \Cr{cnst.divided.derivative.diff} Nh, \quad
              j\in\{1,2\}
            \end{equation}
            for $s = \max(s_1,s_2)$.  In addition, since the columns
            of $\UU_1$ and $\UU_2$ have unit length, we also have that
            \begin{equation}\label{eq.norm.U}
              \|\UU_j\| \leq
              \|\UU_j\|_F \le \sqrt{s},\quad j\in\{1,2\}.
            \end{equation}
            We now represent $\vv$ using the basis
            $\WW_1$ and $\uv$ using the basis $\WW_2$ as follows:
            $$
            \vv = \WW_1 \av, \tab \uv = \WW_2 \bv.
            $$  
            Then using \eqref{eq.norm.error} and \eqref{eq.norm.U},
            and assuming $Nh<1$, we get that
            \begin{align}\label{eq.inner.prod.bound}
              \begin{split}
                \left|\langle \vv,\uv\rangle \right| &=
                \left|\bv^H \WW_2^H \WW_1 \av\right|\\
                &= \left| \bv^H (\UU_2 +
                  \EE_2)^H(\UU_1 + \EE_1) \av\right|\\
                & \le |\bv^H\UU_2^H \UU_1 \av| + |\bv^H \UU_2^H\EE_1
                \av|+|\bv^H \EE_2^H\UU_1 \av| + |\bv^H
                \EE_2^H\EE_1 \av|\\
                & \le |\bv^H\UU_2^H \UU_1 \av| + \|\bv\|\|\av\| s
                \Cr{cnst.divided.derivative.diff} Nh + \|\bv\|\|\av\|
                s \Cr{cnst.divided.derivative.diff} Nh +
                \|\bv\|\|\av\| s
                \Cr{cnst.divided.derivative.diff}^2 (Nh)^2\\
                &\le |\bv^H\UU_2^H \UU_1 \av| +
                \Cl{orth.intermidiate.1}(s) \|\bv\|\|\av\| N h.
              \end{split}
            \end{align}

		By Proposition \ref{prop.sing.limit} we have that
		\begin{equation}\label{eq.UU1.UU2}
\sigma_{\min}(\UU_1)\ge \Xi(s), \tab \sigma_{\min}(\UU_2)\ge \Xi(s),
\end{equation} provided $N \ge
\Cr{const.limit.matrix.N}(s)$, where $\Xi$ and
$\Cr{const.limit.matrix.N}$ are the constants defined\footnote{Here we used the fact the
  $\Cr{const.limit.matrix.N}(s)$ is increasing in $s$ and $\Xi(s)$ is
  decreasing in $s$ (see Proposition \ref{prop.sing.limit}).} in Proposition
\ref{prop.sing.limit}.

Assume that $N$ and $h$ satisfy
\begin{equation}\label{eq.srf.small.2}
  N h \le
  \min\left(1,\frac{\Xi(s)}{2\sqrt{s}
    \Cr{cnst.divided.derivative.diff}}\right)=\Cr{const.high}(s).
\end{equation}

		Then, using \eqref{eq.W1.W2.as.pretubation.of.U1.U2},
\eqref{eq.norm.error}, \eqref{eq.srf.small.2} and \eqref{eq.UU1.UU2},
and applying the standard singular value perturbation bound, we get
		\begin{align}\label{eq.bound.WW1.WW2}
			\begin{split}
\sigma_{\min}(\WW_1)&=\sigma_{\min}(\UU_1 + \EE_1)\ge
\sigma_{\min}(\UU_1) - \|\EE_1\| \ge \frac{\Xi(s)}{2},\\
\sigma_{\min}(\WW_2)&=\sigma_{\min}(\UU_2 + \EE_2)\ge
\sigma_{\min}(\UU_2) - \|\EE_2\| \ge \frac{\Xi(s)}{2}.
			\end{split}
		\end{align}

		By \eqref{eq.bound.WW1.WW2} we conclude that
		\begin{align}\label{eq.orth.norm.a.b}
                  \begin{split}
                    \|\av\| &\le
                  \|\WW_1^{\dagger}\| \|\vv\| \leq 2\Xi^{-1}(s), \\
                  \|\bv\| &\le \|\WW_2^{\dagger}\| \|\uv\|  \leq 2\Xi^{-1}(s).
                \end{split}
		\end{align}
		
		\smallskip
		
		Now combining \eqref{eq.inner.prod.bound},
                \eqref{eq.orth.norm.a.b} and
                \eqref{eq:limit-inner-prod}, we get that
\begin{align*}
  \left|\langle\vv,\uv\rangle\right| &\le |\bv^H\UU_2^H \UU_1 \av|
                                       +4\Cr{orth.intermidiate.1}\Xi^{-2}(s) Nh \\
                                     &\le \frac{4 s \Xi^{-2}(s)\Cr{prod.limit.final}(s_1,s_2)}{N \theta}
                                       + 4\Cr{orth.intermidiate.1}\Xi^{-2}(s)Nh,
\end{align*}
which proves \eqref{eq.oth.main} with
$\Cr{inner.prod.clusters.N} = 4 s\Xi^{-2}(s) \Cr{prod.limit.final}(s_1,s_2)$, and
$\Cr{inner.prod.clusters.srf} = 4\Cr{orth.intermidiate.1}\Xi^{-2}(s)$.
	    
Collecting the assumptions we have made along the proof, regarding the
range of $N$ for which the intermediate claims hold, we required:
\begin{itemize}
\item $N\le \frac{\Cr{const.high}}{h}$, this assumption was used in
\eqref{eq.srf.small.2} in order to establish \eqref{eq.orth.norm.a.b}
and \eqref{eq.inner.prod.bound}.
\item $N \ge \Cr{const.limit.matrix.N}(s)$, this assumption was used to
  establish \eqref{eq.UU1.UU2}.
\end{itemize}

Therefore, we have proved Theorem \ref{thm.orth.spaces} with $\Cr{const.high}$, $\Cr{subsapce.angle.srf}$  and
$\Cr{subsapce.angle.N}$ as above and
with $\Cr{const.low} = \Cr{const.limit.matrix.N}(s)$. \qed

\section{Proof of Theorem \ref{thm:single.cl}}\label{sec:single.cluster}

Given nodes $\xv=\left\{x_1,\dots,x_s\right\}$ and $N=2M$ where
$M$ is an integer, define
\begin{equation*}
  \VVV_N(\xv):=\frac{1}{\sqrt{N}} \VV_N
\times \diag \left\{e^{-\imath M x_j}\right\}_{j=1,\dots,s} =
\frac{1}{\sqrt{2M}} \left[\exp\left(\imath k
x_j\right)\right]^{j=1,\dots,s}_{k=-M,\dots,M}.
\end{equation*}

  Also, let
  $$
  \GG_N=\GG_N(\xv) := \VVV_N(\xv)^H \VVV_N(\xv) = {1\over{2M}} \left[
\mathcal{D}_M \left(x_i-x_j\right)\right]_{i,j},
  $$
  where $\mathcal{D}_M$ is the Dirichlet kernel of order $M$:
  $$
  \mathcal{D}_M(t):=\sum_{k=-M}^M\exp(\imath k t) = \begin{cases}
\frac{\sin\left(\left(M+{1\over 2}\right)t\right)}{\sin{t\over 2}} & t
\notin 2\pi\mathbb{Z}, \\ 2M+1 & \text{else.}
  \end{cases}
  $$
Therefore,
  \begin{equation}\label{eq:sing.single.via.eigs}
  \sigma_j\left(\VV_N\left(\xv\right)\right) = N^{{1\over 2}}
\sigma_j\left(\VVV_N\left(\xv\right)\right) = \sqrt{N \lambda_j
\left(\GG_N\right)},\quad j=1,\dots,s.
\end{equation}

Put $\varepsilon:=Mh=\frac{Nh}{2}$. Further, put
$\left\{y_1,\dots,y_s\right\}=\yv := {1\over h}\xv$. Then we have
\begin{equation*}
  \GG_N = {1\over {2M}}\left[ \mathcal{D}_M
    \left(\varepsilon\frac{\left(y_i-y_j\right)}{M}\right)\right]_{1\leq i,j \leq s},
\end{equation*}
where $\tau \leq |y_i-y_j| \leq 1$ for $i\neq j$.

The following is essentially a variation of \cite[Theorem
8]{wathen2015}, suitable for our setting.

Denote by $\DD=\DD(\yv)$ the distance matrix
$\DD = \left[y_i-y_j\right]_{i,j}$, and
$\DD^k=\left[\left(y_i-y_j\right)^k\right]_{i,j}$ the element-wise
powers of $\DD$.

Next, define for $m=0,1,\dots,s-1$ the $(m+1)\times s$ Vandermonde
matrices
$$
\PP_m = \PP_m(\yv) = \left[y_j^k\right]_{k=0,\dots,m}^{j=1,\dots,s}.
$$
Since the elements of $\{y_j\}$ are pairwise different, the matrices $\PP_m$
have full row rank (equal to $m+1$). Thus, $\dim\ker\PP_m = s-1-m$, and
furthermore, with $\PP_{-1}:=\boldsymbol{0}^H$,
$$
\left\{\boldsymbol{0}\right\} = \ker \PP_{s-1} \subset \ker \PP_{s-2} \subset \dots \subset \ker \PP_{0} \subset \ker \PP_{-1} \equiv \RR^s.
$$

The following key result is precisely the well-known Micchelli lemma.

\begin{lemma}[Lemma 3.1 in \cite{micchelli1986}]\label{lemma.micchelli}
  Let $m=0,1,\dots,s-1$. If $\av \in \ker \PP_{m-1}$ then
  \begin{equation}\label{eq:micchelli}
    (-1)^m \av^H \DD^{2m} \av \geq 0,
  \end{equation}
while equality holds if and only if $\av \in \ker \PP_{m}$.
\end{lemma}

\begin{corollary}
  Let $m=0,1,\dots,s-1$. For each $\av\in\ker\PP_{m-1}$ and $\bv\in\ker\PP_{m}$ we have
  \begin{equation}
    \label{eq:zero-on-kerPPm}
    \av^H\DD^{2m}\bv=0.
  \end{equation}
\end{corollary}
\begin{proof}
  By Lemma \ref{lemma.micchelli} we have that $(-1)^m \DD^{2m}$ is
  positive semi-definite on $\ker\PP_{m-1}$, i.e., if $\PM$ is the
  orthogonal projection matrix onto $\ker\PP_{m-1}$ then
  $\tilde\DD=\PM^H (-1)^m \DD^{2m}\PM$ is positive
  semi-definite. Invoking Lemma \ref{lemma.micchelli} again we obtain
  $\bv^H\tilde{\DD}\bv=0$, implying $\tilde{\DD}\bv=0$.
\end{proof}

It can be readily checked that the Taylor
expansion of the normalized Dirichlet kernel at the origin is
$$
{1\over {2M}}\mathcal{D}_M\left({t\over M}\right) = \sum_{k=0}^\infty
(-1)^{k} \frac{F(M,k)}{(2k)!} t^{2k}, \quad F(M,k):={1\over {2
    M^{2k+1}}}\sum_{m=-M}^M m^{2k}.
$$

Note that by \eqref{eq.integer.power.sum} we have
\begin{equation}\label{eq:fmk-estimate-faulh}
F(M,k) \in \left[\frac1{2k+1}, 1\right].
\end{equation}

We get
\begin{equation*}
	\av^H \GG_N \av = \sum_{k=0}^\infty
	(-1)^{k} \varepsilon^{2k}\frac{F(M,k)}{(2k)!}  \av^H \DD^{2k}\av.
\end{equation*}
Applying the Courant-Fischer minimax principle (see e.g. Theorem 4.2.6 in \cite{horn_matrix_2012}), we get for $m~\in~\{0,1,\dots,s-1\}$
\begin{equation*}
\begin{aligned}
	\lambda_{m+1}(\GG_N) &=  \min_{L:\dim L = s-m}
	\; \;\max_{\|\vv\|=1,\;\vv\in L} \vv^H \GG_N\vv \\&\leq  \max_{\|\av\|=1,\;\av\in \ker \PP_{m-1}} \av^H \GG_N\av \leq \max_{\|\av\|=1,\;\av\in \ker \PP_{m-1}} \sum_{k=m}^{\infty}
	\frac{\varepsilon^{2k}}{(2k)!}\left| \av^H \DD^{2k} \av \right|.
	\end{aligned}
\end{equation*}

For every $k\in\NN$, the entries of $\DD^{2k}$ are bounded from above
by 1, therefore equation \eqref{eq.matnorm} yields 
\[
|\av^H \DD^{2k} \av|\leq \|\DD^{2k}\|\leq \|\DD^{2k}\|_F \leq s.
\] 
Now suppose that $\varepsilon < 1$, then clearly
\begin{equation} \label{eq:eig.upper.final}
\lambda_{m+1}\left(\GG_N\right) \leq \max_{\|\av\|=1,\;\av\in \ker \PP_{m-1}} \sum_{k=m}^{\infty} \frac{\varepsilon^{2k}}{(2k)!} |\av^H \DD^{2k} \av|
\leq se\varepsilon^{2m}. 
\end{equation}

Let $\ker \PP_{m-1} = \ker \PP_{m} \oplus \MS_m$ (i.e. $\MS_m$ is the
orthogonal complement of $\ker \PP_{m}$ in $\ker \PP_{m-1}$). Clearly
$\dim \MS_m=1$. Let $\mathcal{Q}_m= \oplus_{k=0}^m \MS_k$.
Using the minimax principle once again yields
\begin{align}\label{eq:lower.bound.eig}
\notag \lambda_{m+1}(\GG_N) &= \max_{L:
	\dim L = m+1} \; \;\min_{\|\vv\|=1,\vv\in L} \vv^H \GG_N \vv \\
&\geq \min_{\av \in \mathcal{Q}_m,\;\|\av\|=1} \av^H \GG_N \av \notag  =: \mu_{m+1}^\varepsilon.
\end{align}

\begin{proposition}
  There exists a constant $\Cl{uniform.lower.eigen}=
  \Cr{uniform.lower.eigen}(\tau, m, s)$ such that
  \begin{equation}
    \label{eq:asymptotic-lb}
    \liminf_{\varepsilon\rightarrow 0}\varepsilon^{-2m}\mu_{m+1}^\varepsilon  \geq \Cr{uniform.lower.eigen} > 0.
  \end{equation}
\end{proposition}

\begin{proof}

  First notice that $\limsup_{\varepsilon\rightarrow
    0}\varepsilon^{-2m}\mu_{m+1}^\varepsilon < \infty$, for example by using the second
  estimate in \eqref{eq:eig.upper.final}. Now define
  $$
  \rho_{m+1}^{\varepsilon}:=\min_{\av \in \MS_m,\;\|\av\|=1} \av^H
  \GG_N \av.
  $$
  Using \eqref{eq:micchelli} we obtain
\begin{align} \label{eq.eigen.mu}
\rho_{m+1}^\varepsilon \geq \min_{\av \in \MS_m,\;\|\av\|=1} \biggl\{\frac{F(M,m)\varepsilon^{2m}}{(2m)!} \left| \av^H \DD^{2m} \av \right| - \sum_{k=m+1}^{\infty}
\frac{\varepsilon^{2k}}{(2k)!}\left| \av^H \DD^{2k} \av \right|\biggr\}.
\end{align}

  Define
  $$
  \mathcal{Y}(\tau,s):=\left\{\yv = \left\{y_1,\dots,y_s\right\}: \quad
    \tau\leq |y_i-y_j| \leq 1 \text{ for } i\neq j\right\}.
  $$

  The following minimum exists:
  $$
  \Cl{uniform.quadr.form} =
  \Cr{uniform.quadr.form}\left(\tau,m,s\right):=\min_{\yv \in
    \mathcal{Y}(\tau,s)}\;\;\min_{\av \in \MS_{m}(\yv),\;\|\av\|=1} \left|
    \av^H \DD^{2m}(\yv) \av \right| > 0.
  $$

Thus, using \eqref{eq:fmk-estimate-faulh} and the second inequality in
\eqref{eq:eig.upper.final}, we can bound the minimum in \eqref{eq.eigen.mu} uniformly over all $\yv$:
\begin{equation}\label{eq:min-over-Mm}
\rho_{m+1}^\varepsilon \geq \varepsilon^{2m} \left(
  \frac{\Cr{uniform.quadr.form}}{(2m+1)!}-se
  \varepsilon^2\right).
\end{equation}

Unfortunately, we cannot in general conclude that for a fixed
$\varepsilon$ we have
$\mu_{m+1}^{\varepsilon}=\rho_{m+1}^{\varepsilon}$ -- since the
minimum over $\mathcal{Q}_m$ is not necessarily attained by a vector
in $\MS_m$. However, this claim holds asymptotically as
$\varepsilon\to 0$. Indeed, let $\av^\varepsilon\in\mathcal{Q}_m$ be a
unit vector with
$(\av^\varepsilon)^H\GG_N \av^\varepsilon =\mu_{m+1}^\varepsilon$. By passing to a
converging subsequence, we can assure that
$\lim_{\varepsilon_k\rightarrow 0} \av^{\varepsilon_k} = \av^*$ and a
standard calculation yields
\begin{align}\label{eq.mu.gn}
	\lim_{\varepsilon_k\rightarrow 0} \varepsilon_k^{-2m}\mu_{m+1}^{\varepsilon_k} = 	\lim_{\varepsilon_k\rightarrow 0} \varepsilon_k^{-2m} (\av^*)^H\GG_N\av^*<\infty.
\end{align} 
We claim that $\av^*\in\MS_m$. Otherwise, let $\ell<m$ be the smallest index such that the projection onto $\MS_\ell$ of $\av^*$ is non-zero. We write
$$
\av^* = \vv +\wv,\qquad \vv\in\MS_{\ell}, ~\wv\in \bigoplus_{k=\ell+1}^m \MS_k.
$$

Calculations analogous to \eqref{eq.eigen.mu} and
\eqref{eq:min-over-Mm} give
\begin{align*}
	(\av^*)^H\GG_N\av^* \geq \varepsilon^{2\ell} \|\vv\|_2^2\left(
	\frac{\Cr{uniform.quadr.form}}{(2\ell+1)!}-se
	\varepsilon^2\right)+2\vv^H \GG_N \wv + \wv^H\GG_N\wv.
\end{align*}
By \eqref{eq:eig.upper.final},
$\wv^H\GG_N\wv \leq se\varepsilon^{2\ell+2}$. By
\eqref{eq:zero-on-kerPPm} we have that
\begin{equation*}
\vv^H \GG_N \wv = \sum_{k=0}^\infty
(-1)^{k} \varepsilon^{2k}\frac{F(M,k)}{(2k)!}  \vv^H \DD^{2k}\wv =  \sum_{k=\ell+1}^\infty
(-1)^{k} \varepsilon^{2k}\frac{F(M,k)}{(2k)!}  \vv^H \DD^{2k}\wv
\end{equation*}
as $\vv \in\ker \PP_{\ell-1}$ and $\wv \in\ker\PP_{\ell}$. That would
result in
$$
(\av^*)^H \GG_N \av^* \geq \tilde{C} \varepsilon^{2\ell},\quad \tilde{C}>0,
$$
a contradiction to the finiteness of the limes
\eqref{eq.mu.gn}. Therefore by \eqref{eq:min-over-Mm}
$$
\liminf_{\varepsilon\to 0}\varepsilon^{-2m}\mu_{m+1}^{\varepsilon} = \liminf_{\varepsilon\to 0}\varepsilon^{-2m}\rho_{m+1}^{\varepsilon} \geq \frac{\Cr{uniform.quadr.form}}{(2m+1)!},
$$
concluding the proof with $\Cr{uniform.lower.eigen}=\frac{\Cr{uniform.quadr.form}}{(2m+1)!}$.
\end{proof}

Using \eqref{eq:asymptotic-lb}, we conclude that there exists
$\varepsilon^*(\tau,m,s)$ such that for all
$\varepsilon<\varepsilon^*$ we have
\begin{equation}
  \label{eq:lower.eig.final}
\lambda_{m+1}\left(\GG_N\right) \geq
\frac12 \Cr{uniform.lower.eigen}(\tau, m, s) \varepsilon^{2m}.
\end{equation}

Combining \eqref{eq:sing.single.via.eigs}, \eqref{eq:eig.upper.final}
and \eqref{eq:lower.eig.final}, we conclude that
\eqref{eq:sing.single.final} holds with
\begin{align*}
  \Cr{single.cluster.Nh}(\tau,s) &:= 2\min\left(1, \min_{0\leq m <s}\varepsilon^*(\tau,m,s)\right),\\
  \Cr{sing.lower.1}(\tau,s)	&:=\min_{0 \leq m < s}2^{-m-\frac12}\Cr{uniform.lower.eigen}(\tau,m,s)^{1\over 2},\\
  \Cr{sing.upper.1}(s)&:= \sqrt{se}.\qedhere
\end{align*}

\begin{remark}
  Unfortunately, the constant $\Cr{single.cluster.Nh}$ could not
  be given explicitly.
\end{remark}

\section{Spectral properties of multi-cluster Vandermonde matrices}\label{sec:multi-cluster-proofs}
	\subsection{Singular values of nearly orthogonal spaces}
	In this section we consider an $N \times s$ matrix $\AAA$ whose columns are 
	partitioned as $\AAA=[\AAA_1,\ldots,\AAA_M]$, and with the blocks $\AAA_j$
	having the following property: Let $L_j$ be the subspace spanned be the columns 
	of the sub-matrix $\AAA_j$. We consider the case where the minimal principle 
	angle between each pair of subspaces $L_j,L_k$ is large: 
	$$\angle_{\min}(L_j,L_k) \ge \frac{\pi}{2} -\alpha,$$      
	for some ``small enough'' $\alpha$. 
	
	\smallskip
	
	We show below that in this case, the singular 
	values of $\AAA$ are given by a multiplicative perturbation of the singular 
	values of all the sub-matrices $\AAA_j$, the size of the multiplicative factor of the 
	perturbation $\gamma$ is $\sqrt{1- s\alpha} \le \gamma \le \sqrt{1+s\alpha}$.
	
	\begin{lemma}\label{lem.sing.orth.spaces}
		Let $\AAA \in \CC^{N\times s}$, $N \ge s$, such that $\AAA$ is given in the following block form 
		$$\AAA=[\AAA_1,\ldots,\AAA_M],$$
		with $\AAA_j \in \CC^{N\times s_j}$ and $\sum_{j=1}^M s_j=s$. Let $L_j \subset \CC^{N}$ be the subspace
		spanned by the columns of the sub-matrix $\AAA_j$. 
		Assume that for all $1\le j,k\le M,\; j\ne k$, and $0 \le \alpha \le \frac{1}{s}$, 
		\begin{equation}\label{eq.orth.space.angle}
			\angle_{\min}(L_j,L_k)\ge \frac{\pi}{2}-\alpha.
		\end{equation} 

                Then the following hold.

                \begin{enumerate}
                \item For each $j=1,\ldots,M$, let
                  $\AAA_j = \QQ_j \RRR_j$ be the QR-decomposition of
                  $\AAA_j$, where $\QQ_j\in \CC^{N\times s_j}$ has
                  orthonormal columns, and
                  $\RRR_j \in \CC^{s_j\times s_j}$ is upper
                  triangular. Write
                  \begin{equation}\label{eq.A.block.QR}
                    \AAA = \QQ\RRR \equiv [\QQ_1,\ldots,\QQ_M]\diag(\RRR_1,\ldots,\RRR_M),
		\end{equation}
		where $\diag(\RRR_1,\ldots,\RRR_M)\in \CC^{s\times s}$
                is a block diagonal matrix whose diagonal blocks are
                $\RRR_1,\ldots,\allowbreak\RRR_M$. Then
                \begin{equation}
                  \label{eq:QQ-sing-vals}
                  \sqrt{1-s\alpha} \leq \sigma_{\min}(\QQ) \leq \sigma_{\max}(\QQ) \leq \sqrt{1+s\alpha}.
                \end{equation}

		\item 		Let 
		$$\sigma_1\ge \ldots \ge \sigma_s$$ 
		be the ordered collection of all the singular values
                of $\AAA$, and let
		$$\tilde{\sigma}_1\ge\ldots\ge \tilde{\sigma_{s}}$$ 
                be the ordered collection of all the singular values
                of the sub-matrices $\{\AAA_j\}$. Then
		\begin{align}
			\sqrt{1-s\alpha}\; \tilde{\sigma}_j \le \sigma_j\le 
			\sqrt{1+s\alpha}\; \tilde{\sigma}_j && j=1,\ldots,s.
		\end{align}
              \end{enumerate}
            \end{lemma}
	
	For the proof of Lemma \ref{lem.sing.orth.spaces} we require the following standard estimate 
	of the singular values of matrix product (see e.g. \cite[Theorem 4.5 and exercise 6 on page 36]{stewart1990matrix}).
	\begin{lemma}\label{lem.sing.matrix.prod}
		For $m\ge n$, Let $\mathbf{C} \in \CC^{m\times n}$, $\mathbf{B}\in\CC^{m\times m}$ and $\mathbf{A}\in\CC^{m\times n}$ 
		such that $\mathbf{C}=\mathbf{B}\mathbf{A}$.
		Then the singular values of $\mathbf{C}$ are given by a multiplicative perturbation
		of the singular values of $\mathbf{A}$ as follows:
		\begin{align}
			\sigma_{\min}(\mathbf{B}) \sigma_j(\mathbf{A}) \le  \sigma_j(\mathbf{C}) \le \sigma_{\max}(\mathbf{B}) \sigma_j(\mathbf{A}), && 1\le
			j\le n \label{eq.lem.matrix.product}.
		\end{align}
	\end{lemma}	      
	\begin{proof}[Proof of Lemma \ref{lem.sing.orth.spaces}]
		
		\smallskip 
		
		First we argue that the singular values of $\RRR=\diag(\RRR_1,\ldots,\RRR_M)$ 
		are given by $\tilde{\sigma}_1\ge\ldots\ge \tilde{\sigma_{s}}.$
		 
		\smallskip
		 
		Indeed we have that the singular values 
		of $\RRR$ are given by the union of the singular values of its 
		diagonal blocks $\RRR_1,\ldots,\RRR_M$. 
		From the other hand, for each $j$, the singular values of $\AAA_j$
		are equal to the singular values of $\RRR_j$ (and this is true since $\QQ_j$ is an orthogonal matrix).
		Therefore the singular values of $\RRR$, ordered according to their magnitude, are exactly 
		\begin{equation}\label{eq.sing.R}
			\tilde{\sigma}_1\ge\ldots\ge \tilde{\sigma_{s}}.
		\end{equation}
		
		\smallskip
		
		Put $\QQ = [\QQ_1,\ldots,\QQ_M]$. Next we show that $\sigma_{\max}(\QQ)\le \sqrt{1+s\alpha}$ and that $\sigma_{\min}(\QQ)\ge \sqrt{1-s\alpha}$. 
		
		\smallskip
		
		We write the Gramian matrix
		$$\QQ^H\QQ=[\QQ_1^H;\ldots;\QQ_M^H][\QQ_1,\ldots,\QQ_M]=
		  \begin{bmatrix}
    		\QQ_1^H\QQ_1 & \QQ_1^H\QQ_2 & \dots & \QQ_1^H\QQ_M \\
    		\QQ_2^H\QQ_1 & \QQ_2^H\QQ_2 & \dots & \QQ_2^H\QQ_1 \\
    		\vdots & \vdots & \vdots & \vdots \\
    		\QQ_M^H\QQ_1 & \QQ_M^H\QQ_2 & \dots & \QQ_M^H\QQ_M
  		\end{bmatrix}.
  		$$
  		The off-diagonal blocks are made out of inner products of unit vectors from different subspaces $L_j$.
  		By \eqref{eq.orth.space.angle}, for each pair of unit vectors $\vv \in L_j$ and $\uv \in L_k$, $j \ne k$,
  		$$|\langle \uv,\vv\rangle| \le \cos\left(\frac{\pi}{2} - \alpha\right)=\sin(\alpha) \le \alpha.$$ 
  		Therefore the absolute value of each entry in the off-diagonal blocks is less than $\alpha$.
  		
  		\smallskip
  		
  		On the other hand, for each $j=1,\ldots,M$, the diagonal block $\QQ_j^H\QQ_j = \II_{s_j}$, where
  		$\II_{s_j}$ is the $s_j\times s_j$ identity matrix. We can therefore write $\QQ^H\QQ$ as 
  		\begin{equation}\label{eq.QQ.pertubation}
  			\QQ^H\QQ = \II_s + \EE,
  		\end{equation}
  		where $\II_s$ is the $s\times s$ identity matrix, and $\EE \in \CC^{s\times s}$ is an Hermitian matrix,
  		the absolute value of each one of its entries is bounded by $\alpha$. Therefore 
  		\begin{align}
			\lambda_{\min}(\mathbf{E})&\geq -\|\EE \|\ge-\|\EE\|_F\geq 
			-s\alpha,\label{eq.lambda.min.E.alpha}\\
			\lambda_{\max}(\mathbf{E})&\leq \|\EE \|\leq \|\EE\|_F\leq s\alpha.\label{eq.lambda.max.E.alpha}
		\end{align} 
		Now using Weyl's perturbation inequality on the perturbation \eqref{eq.QQ.pertubation} and
		the bounds \eqref{eq.lambda.min.E.alpha} and \eqref{eq.lambda.max.E.alpha}, we have that
		\begin{align}
			\sigma_{\min}(\QQ)&=\sqrt{\lambda_{\min}\left(\QQ\right)}\ge \sqrt{1-s\alpha}\label{eq.QQ.sing.min},\\
			\sigma_{\max}(\QQ)&=\sqrt{\lambda_{\max}\left(\QQ\right)}\le \sqrt{1+s\alpha}\label{eq.QQ.sing.max}.
		\end{align}
		 
		 We conclude that according to \eqref{eq.A.block.QR}, \eqref{eq.sing.R}, \eqref{eq.QQ.sing.min} 
		 and \eqref{eq.QQ.sing.max}, $\AAA$ can be written as follows:
		 $$\AAA = \QQ \RRR,$$
		 where the minimal and maximal singular values of $\QQ$ are bounded as in \eqref{eq.QQ.sing.min} 
		 and \eqref{eq.QQ.sing.max}, and the singular values of $\RRR$ are exactly 
		 $\tilde{\sigma}_1\ge\ldots\ge \tilde{\sigma_{s}}$. The proof of Lemma \ref{lem.sing.orth.spaces}
		 is then completed by invoking Lemma \ref{lem.sing.matrix.prod} with $\mathbf{C}=\AAA$, 
		 $\mathbf{B}=\QQ$ and $\AAA = \RRR$ (on the left side are the matrices of Lemma \ref{lem.sing.matrix.prod}).    
               \end{proof}

               \subsection{Multi-cluster subspace angles}

               \begin{proposition}[Multi-cluster subspace angles]\label{prop:mult.clust.angl}
                 Suppose that
                 $\xv=\left\{x_1,\dots,x_s\right\} \subset (-\pi,\pi]$
                 forms an
                 $((h^{(j)}, s^{(j)})_{j=1}^M, \theta)$-clustered
                 configuration, and put $h=\max_{j}(h^{(j)})$.  Then
                 there exist constants $\Cr{multi.cluster.N.theta}$, $\Cr{multi.cluster.N.h}$,
                 $\Cl{max.N.theta}$ and $\Cl{max.N.h}$, depending only
                 on $s^{(1)},\ldots,s^{(M)}$, such that for
                 $\frac{\Cr{multi.cluster.N.theta}}{\theta} \le N \le
                 \frac{\Cr{multi.cluster.N.h}}{h}$ we have
                 \begin{equation}\label{eq.angle.j.k.spaces}
                   \angle_{\min}(L(\cv^{(j)},N),L(\cv^{(k)},N))
                   \ge \frac{\pi}{2} - \alpha,\quad \alpha:=\frac{\Cr{max.N.theta}}{N\theta} +
                   \Cr{max.N.h}Nh \leq {1\over s}, \tab 1\le j<k\le M.
		\end{equation}
              \end{proposition}
              \begin{proof}
		\eqref{eq.angle.j.k.spaces} immediately follows from
                Theorem \ref{thm.orth.spaces} with
		\begin{align}
                  \Cr{max.N.theta}&:=\max_{1 \le j<k\le M}	\Cr{subsapce.angle.N}(s^{(j)},s^{(k)}),\\
                  \Cr{max.N.h}&:=\max_{1 \le j<k\le M}	\Cr{subsapce.angle.srf}(s^{(j)},s^{(k)}),\\
                  \Cr{multi.cluster.N.theta} &:= \max\left(\pi \left\{\max_{1 \le j<k\le M} \Cr{const.low}(s^{(j)},s^{(k)})\right\},2 s \Cr{max.N.theta}\right), \\
                  \Cr{multi.cluster.N.h} &:= \min\left(\left\{\min_{1 \le j<k\le M} \Cr{const.high}(s^{(j)},s^{(k)})\right\},\frac{1}{2s\Cr{max.N.h}}\right).
		\end{align}
		Here $\Cr{const.low}$, $\Cr{const.high}$, $\Cr{subsapce.angle.srf}$ and $\Cr{subsapce.angle.N}$ are the constants specified in Theorem \ref{thm.orth.spaces}.
	\end{proof}
               
 	\subsection{Proof of Theorem \ref{thm.union}}\label{sec:proof.orth}
        Let $\xv$ form an
        $((h^{(j)}, s^{(j)})_{j=1}^M, \theta)$-clustered
        configuration and put $h=\max_{j}(h^{(j)})$. Let
        $\sigma_1\ge \ldots \ge \sigma_s$ and
        $\tilde{\sigma}_1\ge\ldots\ge\tilde{\sigma_{s}}$ be as
        specified in Theorem \ref{thm.union}. Without loss of
        generality we assume that $\VV_N(\xv)$ is organized in block
        form, according to the clusters, as follows:
        \begin{equation}\label{eq:Vn-block-form}
          \VV_{N}(\xv)=\left[\VV_N(\cv^{(1)}),\ldots,\VV_N(\cv^{(M)})\right].
        \end{equation}
        By Proposition \ref{prop:mult.clust.angl} we have the
                estimate \eqref{eq.angle.j.k.spaces} for
		\begin{equation*}
			\frac{\Cr{multi.cluster.N.theta}}{\theta} \le 
			N\le \frac{\Cr{multi.cluster.N.h}}{h},
		\end{equation*}
                when furthermore $\alpha \le \frac{1}{s}$.  Now we
                invoke Lemma \ref{lem.sing.orth.spaces} with
                $\AAA = \VV_{N}(\xv)$, $\AAA_j = \VV_N(\cv^{(j)})$ and
                $\alpha$ as above and get that
		\begin{align*}
                  \left(1-s\alpha\right)^{\frac{1}{2}} \tilde{\sigma}_j \le \sigma_j\le 
                  \left(1+s\alpha\right)^{\frac{1}{2}} \tilde{\sigma}_j && j=1,\ldots,s,
		\end{align*}
		thus proving Theorem \ref{thm.union} with 
		$\Cr{multi.cluster.N.theta}$,
		$\Cr{multi.cluster.N.h}$ as above and 
		$\Cr{vandermonde.union.N}=\Cr{max.N.theta}s$ and 
		$\Cr{vandermonde.union.N.h}=\Cr{max.N.h}s$.
		
		\begin{remark} \label{remark.const.maxcluster}
			Let $s_{\max} = \max( s^{(j)})$. Then  $\Cr{max.N.theta},\Cr{max.N.h}$ depend only on $s_{\max}$ and not on $s$, while $\Cr{multi.cluster.N.theta}$ scales linearly in $s$ and $\Cr{multi.cluster.N.h}$ linearly in $\frac1s$. Finally, $\Cr{vandermonde.union.N}$ and 
			$\Cr{vandermonde.union.N.h}$ scale linearly in $s$. Thus, the scaling of the constants in the total number of nodes is only linear, while the scaling in the largest cluster size is more severe. For example, our estimates give 
			\begin{align*}
			\Cr{vandermonde.union.N}&\leq 2\pi^2s(2s_{\max}-1)s_{\max}\Xi^{-2}(s_{\max})\\
			\Cr{vandermonde.union.N.h} &= 8\pi s(2+\sqrt{2})s_{\max}\Xi^{-2}(s_{\max}),
			\end{align*}
			with $\Xi$ as in Proposition \ref{prop.sing.limit}.
		\end{remark}

                \subsection{Proof of Theorem
                  \ref{thm:ls-accuracy}}\label{sec:proof.ls}


                Let $\AAA\in\CC^{m\times n}$ be a matrix. We use the
                following notations:
                \begin{itemize}

              \item $\|\AAA\|_{k,1}$ for the $\ell_1$ norm of the
                $k$-th row of $\AAA$, i.e.
                  $$
                  \|\AAA\|_{k,1}:=\sum_{\ell=1}^{n} \bigl|
                  \left(\AAA\right)_{k,\ell} \bigr|,\quad
                  k\in\left\{1,\dots,m\right\};
                  $$
                \item $\|\AAA\|_{k,\max}$ for the maximum norm of the
                $k$-th row of $\AAA$, i.e.
                  $$
                  \|\AAA\|_{k,\max}:=\max_{1\leq \ell \leq n} \bigl|
                  \left(\AAA\right)_{k,\ell} \bigr|,\quad
                  k\in\left\{1,\dots,m\right\}.
                  $$ 
                \end{itemize}

                \begin{lemma}
                  Let $\AAA=\BBB\CCC \in \CC^{m\times n}$, where
                  $\CCC\in\CC^{p\times n}$. Then
                  \begin{equation}
                    \label{eq:row-norm-for-product}
                    \|\AAA\|_{k,1} \leq \sqrt{pn} \|\BBB\|_{k,\max} \|\CCC\|_F.
                  \end{equation}
                \end{lemma}
                \begin{proof}
                  We have
                  \begin{align*}
                    \|\AAA\|_{k,1} &= \sum_{\ell=1}^n \bigl|
                                     \left(\AAA\right)_{k,\ell} \bigr| = \sum_{\ell=1}^n  \left|\sum_{j=1}^p\left(\BBB\right)_{k,j} \left(\CCC\right)_{j,\ell} \right| \\
                                   & \leq \|\BBB\|_{k,\max} \sum_{j,\ell}\left|\left(\CCC\right)_{j,\ell}\right| \\
                                   &\leq \sqrt{pn} \|\BBB\|_{k,\max} \|\CCC\|_F,
                  \end{align*}
                  where the last transition is just the H\"older's inequality.
                \end{proof}

                First we prove the following estimate.
\begin{proposition}[Pseudoinverse row norms] \label{prop:pseudo.row.norms}
  Suppose that the node set
  $\xv=\left\{x_1,\dots,x_s\right\} \subset (-\pi,\pi]$ forms an
  $((h^{(j)}, \tau^{(j)},s^{(j)})_{j=1}^M, \theta)$-clustered
  configuration, with $h=\max_{1\leq j\leq M}(h^{(j)})$. Then there
  exist constants $\Cr{ls.N.theta}(s^{(1)},\dots,s^{(M)})$ and
  $\Cr{ls.N.h}(s^{(1)},\dots,s^{(M)})$ such that for all
  $\frac{\Cr{ls.N.theta}}{\theta} \leq N \leq
  \frac{\Cr{ls.N.h}}{h}$ and each $j\in\left\{1,2,\dots,M\right\}$
  there exists a constant $\Cr{ls.ub}(s^{(j)},\tau^{(j)})$ such that
  for all $\ell\in C_j\left(\xv\right)$ (recall Definition
  \ref{def.cluster.index}) we have
    \begin{equation}\label{eq:pinv.bound}
  \|\VV_N^{\dagger}(\xv)\|_{\ell,1}
\leq \Cr{ls.ub} s \left(\frac{1}{Nh^{(j)}}\right)^{s^{(j)}-1}.
\end{equation}
\end{proposition}
\begin{proof}
  Again, applying Proposition \ref{prop:mult.clust.angl} and then Lemma
  \ref{lem.sing.orth.spaces} to the matrix $\VV_N(\xv)$ assumed to be
  in the block form \eqref{eq:Vn-block-form}, we obtain the block
  QR-decomposition \eqref{eq.A.block.QR} $\VV_N(\xv)=\QQ\RRR$. Since
  $\QQ\in\CC^{(N+1)\times s}$ has full column rank, and
  $\RRR\in \CC^{s\times s}$ is invertible, we have
  \begin{equation}\label{eq:vn-pinv-as-prod}
  \VV_N^{\dagger} = \RRR^{-1} \QQ^{\dagger}.
\end{equation}

Using \eqref{eq:QQ-sing-vals}, we have
\begin{equation}\label{eq:qq-pinv-spect-norm}
 \|\QQ^{\dagger}\| = \sigma_{\max}\left(\QQ^{\dagger}\right) = \sigma^{-1}_{\min} (\QQ) \leq (1-s\alpha)^{-{1\over 2}},
\end{equation}
where $\alpha = \frac{\Cr{max.N.theta}}{N\theta} + \Cr{max.N.h}Nh$ as
provided by Proposition \ref{prop:mult.clust.angl}.

On the other hand, each one of the blocks $\RRR_j$ has its singular
values exactly equal to the singular values of
$\VV_N\left(\cv^{(j)}\right)$. By Theorem \ref{thm:single.cl}, the
smallest one scales like
$N^{1\over 2} \left(Nh^{(j)}\right)^{s^{(j)}-1}$, and therefore for
some constant $\Cl{lb.R}(s^{(j)},\tau^{(j)})$ we have, using equation \eqref{eq.matnorm},
$$
\|\RRR_j^{-1}\|_{\max} \leq \|\RRR_j^{-1}\| = \sigma_{\min}^{-1}(\RRR_j) \leq \Cr{lb.R} {1\over\sqrt{N}}\left(\frac{1}{Nh^{(j)}}\right)^{s^{(j)}-1}.
$$

Let $\ell\in C_j\left(\xv\right)$, then by
\eqref{eq:row-norm-for-product} applied to \eqref{eq:vn-pinv-as-prod},
\eqref{eq:qq-pinv-spect-norm} and the fact that $\QQ^{\dagger}$ is of
rank $s$, we have that
 \begin{align*}
  \|\VV_N^{\dagger}\|_{\ell,1} & \leq \sqrt{sN} \|\RRR_j^{-1}\|_{\max} \|\QQ^{\dagger}\|_{F} \\
                               &\leq s\Cr{lb.R} (1-s\alpha)^{-{1\over 2}} \left(\frac{1}{Nh^{(j)}}\right)^{s^{(j)}-1}.
\end{align*}
Here, we used equation \eqref{eq.matnorm} once again.
If $N\theta > 4s\Cr{max.N.theta}$ and $Nh < \frac{1}{4s\Cr{max.N.h}}$ we
have $\alpha < \frac{1}{2s}$ and consequently
$\left(1-s\alpha\right)^{-{1\over 2}} < \sqrt{2}$. This
completes the proof of \eqref{eq:pinv.bound} with $\Cr{ls.N.theta}(s^{(1)},\cdots,s^{(M)}) =
\max\left(\Cr{multi.cluster.N.theta},4s\Cr{max.N.theta}\right)$,
$\Cr{ls.N.h}(s^{(1)},\dots,s^{(M)}) =
\min\left(\Cr{multi.cluster.N.h},\frac{1}{4s\Cr{max.N.h}}\right)$ and
  $\Cr{ls.ub}(s^{(j)},\tau^{(j)}) = \sqrt{2}\Cr{lb.R}$.
\end{proof}

\begin{proof}[Proof of Theorem \ref{thm:ls-accuracy}]
From Definition \ref{def.ls.solution} we clearly have
$\av\left(\xv,\bv\right) = \VV_N^{\dagger}\left(\xv\right) \bv$, and
since $\av_0=\VV_N^{\dagger}(\xv)\bv_0$, we obtain by \eqref{eq:pinv.bound}
\begin{align*}
 \left|\left(\av-\av_0\right)_{\ell}\right|  = 
 \left| \left(\VV_N^{\dagger}\left(\bv-\bv_0\right)\right)_{\ell}\right|
 & \leq \|\VV_N^{\dagger}(\xv)\|_{\ell,1}\|\bv-\bv_0\|_{\infty} \\
 & \leq \Cr{ls.ub} s \left(\frac{1}{Nh^{(j)}}\right)^{s^{(j)}-1} \|\bv-\bv_0\|_{\infty}.
\end{align*}
This completes the proof of \eqref{eq:ls-componentwise-bound} with
$\Cr{ls.N.theta}$, $\Cr{ls.N.h}$ and
$\Cr{ls.ub}$ as in Proposition \ref{prop:pseudo.row.norms}.
\end{proof}

\section{Numerical experiments}\label{sec:numerics}
In this section we provide basic numerical evidence supporting our
main results. All calculations were performed using Julia 1.1 with
standard packages, in double precision floating point (and sometimes
multi-precision).

\subsection{Cluster subspace angles}
We use the notation from Theorem \ref{thm.orth.spaces}. In order to
compute the minimal principal angle $\angle_{\min}(L_1,L_2)$ between
$L_1=L(\xv,N)$ and $L_2=L(\yv,N)$, we use the standard SVD-based
algorithm (see e.g. \cite{bjorck1973,knyazev2002}) which is
numerically stable for large angles. We then compute the
\emph{complementary angle}
$$
\beta\left(L_1,L_2\right):={\pi\over 2} - \angle_{\min}(L_1,L_2).
$$
In the experiments, the two clusters were chosen to consist of
equispaced nodes with the same cluster size $h^{(1)}=h^{(2)}=h$, and
with a prescribed distance $\theta$ between the closest nodes.

According to the estimate \eqref{eq.thm.angle} from Theorem
\ref{thm.orth.spaces}, we have the bound
\begin{equation}\label{eq:complementary-angle-bound}
  \beta\left(L_1,L_2\right) \leq {\Cr{subsapce.angle.N}\over {N\theta}} +  \Cr{subsapce.angle.srf} Nh.
\end{equation}

In the first set of experiments, we kept the value of $\theta$ fixed,
and changed $N,h$ simultaneously so that the product $Nh$ remained
fixed. We chose 3 different values $Nh=10^{-10},10^{-5},0.1$. The
dependence of $\beta$ on $N$ is presented in Figure
\ref{fig:angle-Nh}. For $s^{(1)}=4, s^{(2)}=2$ (left plot) the
asymptotic decay $\beta \sim {1\over N}$ is clearly seen, with smaller
values of $\beta$ corresponding to smaller values of $Nh$. This
suggests that ${\Cr{subsapce.angle.N}\over {N\theta}}$ is indeed the
dominant term with respect to $\Cr{subsapce.angle.srf} Nh$ in
\eqref{eq:complementary-angle-bound}. However, for $s^{(1)}=s^{(2)}=6$
(right plot) we see that when $Nh=10^{-5}$, the value of $\beta$ is
relatively constant, $\approx 10^{-5}$, while the value of $\beta$ for
$Nh=0.1$ decays relatively slowly with $N$ (for $Nh=10^{-10}$ we still
have $\beta\sim{1\over N}$). This suggests that the value of $Nh$ is
indeed important for controlling the subspace angle in the regime
$Nh\ll 1$, however for $Nh =\MO(1)$ there must be other factors.

\begin{figure}
  \centering
  \includegraphics[width=0.45\linewidth]{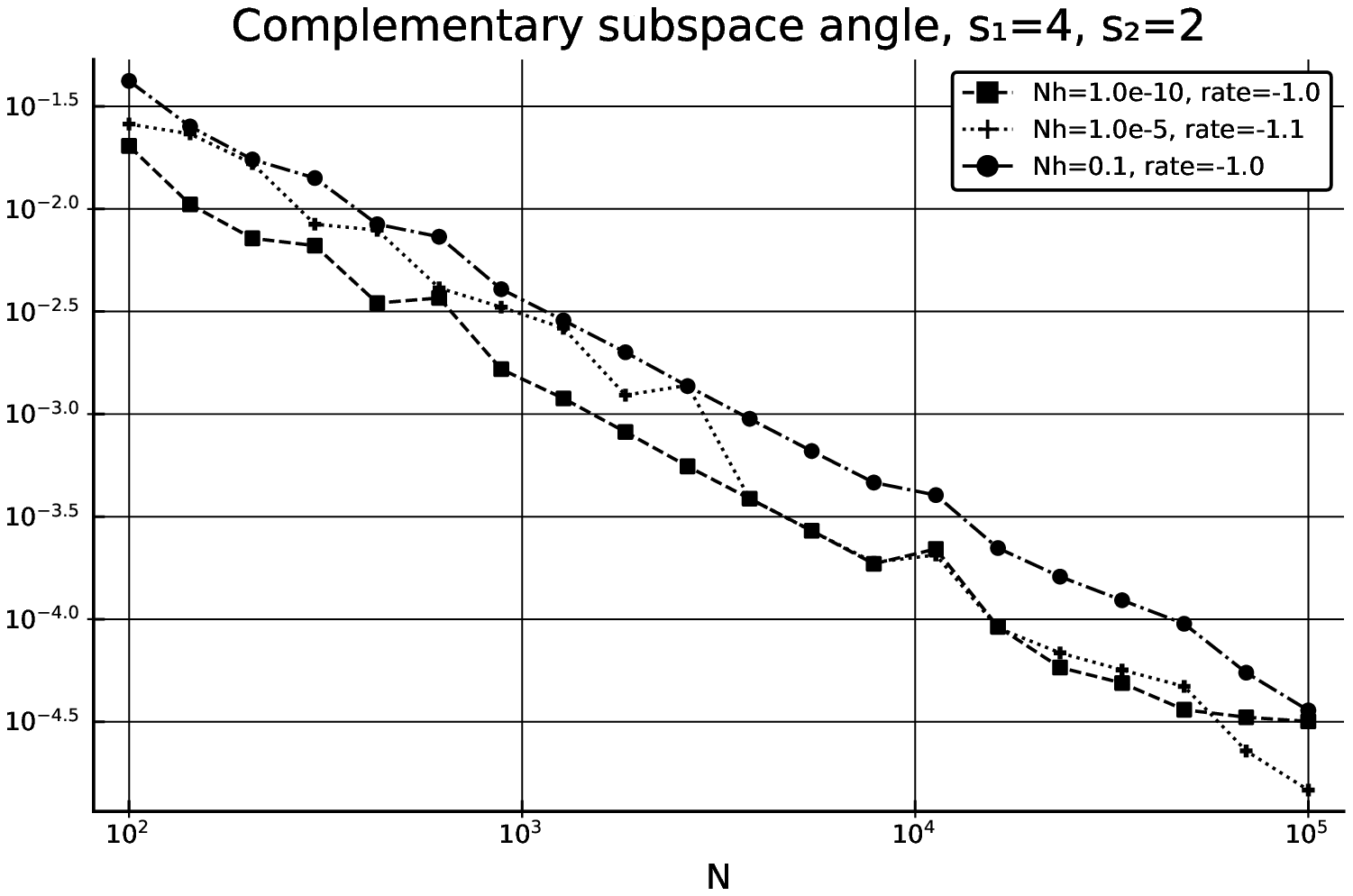}   \includegraphics[width=0.45\linewidth]{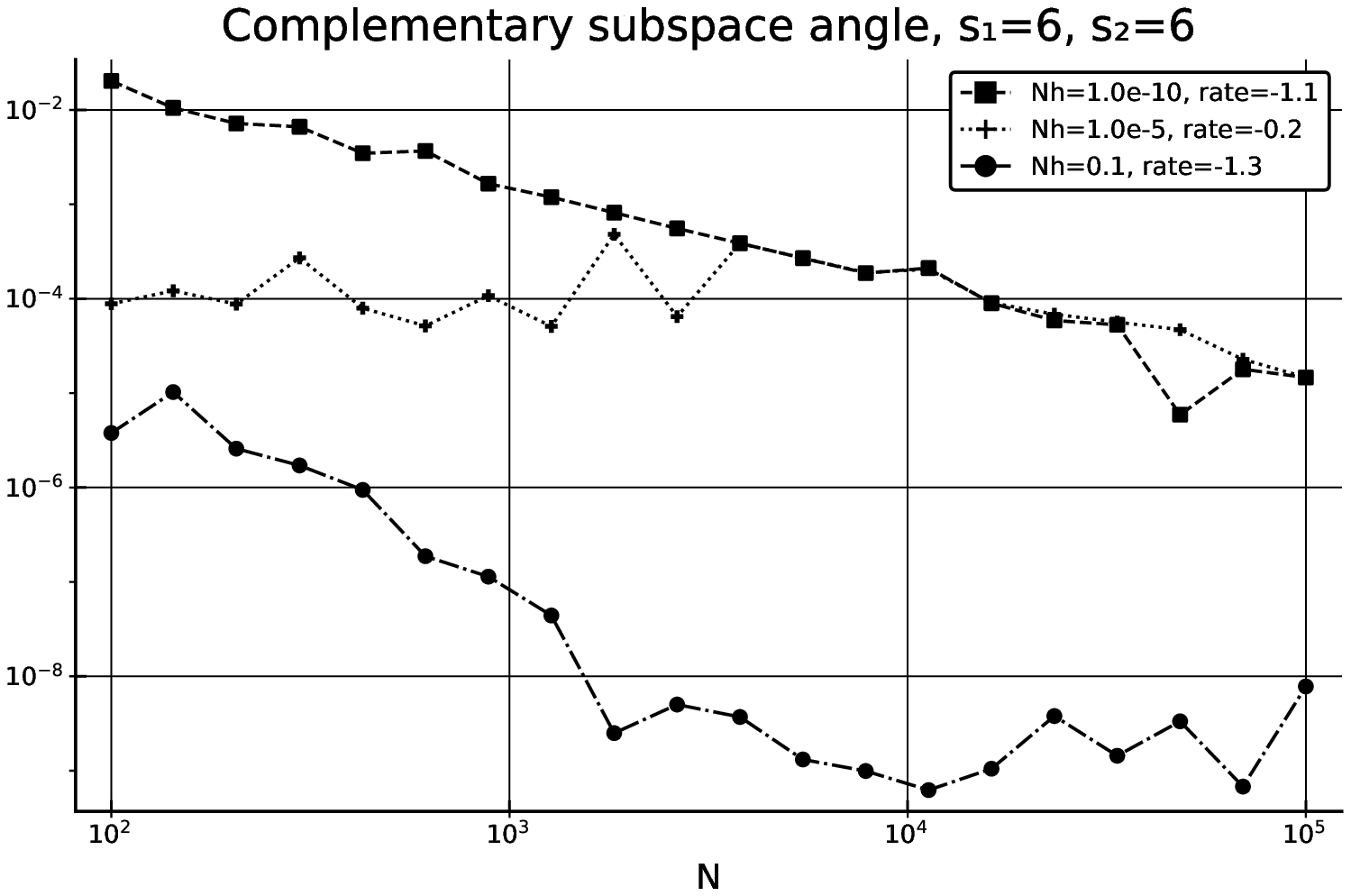}
  \caption{Complementary subspace angle $\beta$. $Nh$ and $\theta$
    fixed, varying $N$.}
  \label{fig:angle-Nh}
\end{figure}

In the second set of experiments, we kept the values of $N$ and
$\theta$ fixed, while changing $h$. We chose again 3 different values
$\theta=0.01,0.1,1$. The dependence of $\beta$ on $h$ (or $Nh$) in
this case is shown in Figure \ref{fig:angle-Ntheta}. Notice that for
small enough $Nh$ we indeed see that $\beta$ approaches a positive
value which is proportional to ${1\over{N\theta}}$, i.e. the dominant
role is played by the cluster separation. For increasing values of
$Nh$, the actual cluster subspaces move further away from the limit
spaces and therefore this regime is not covered by our
theory. However, apparently also in this case $\beta$ remains small,
but this must be due to other factors.

\begin{figure}
  \centering
  \includegraphics[width=0.45\linewidth]{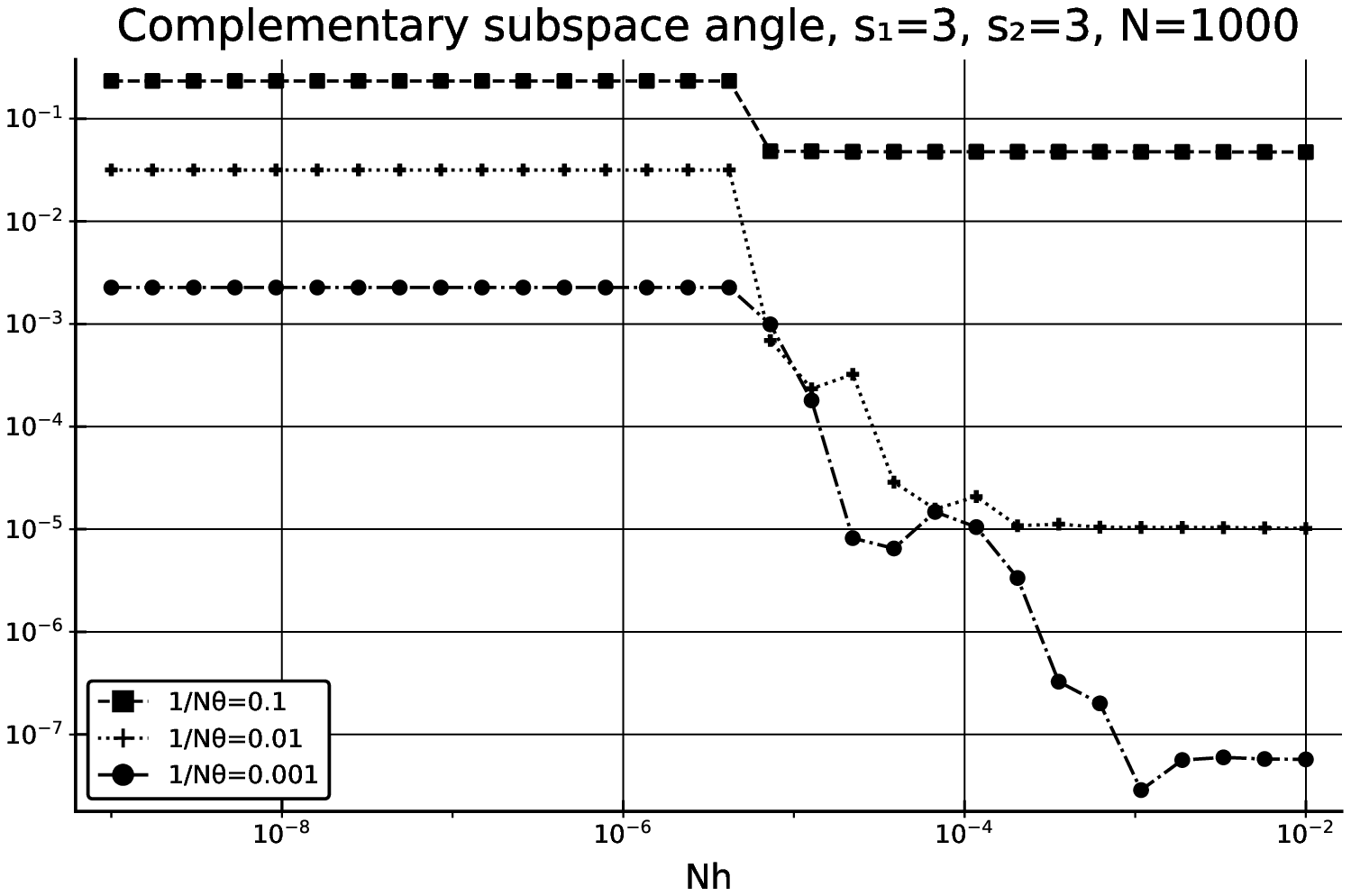}   \includegraphics[width=0.45\linewidth]{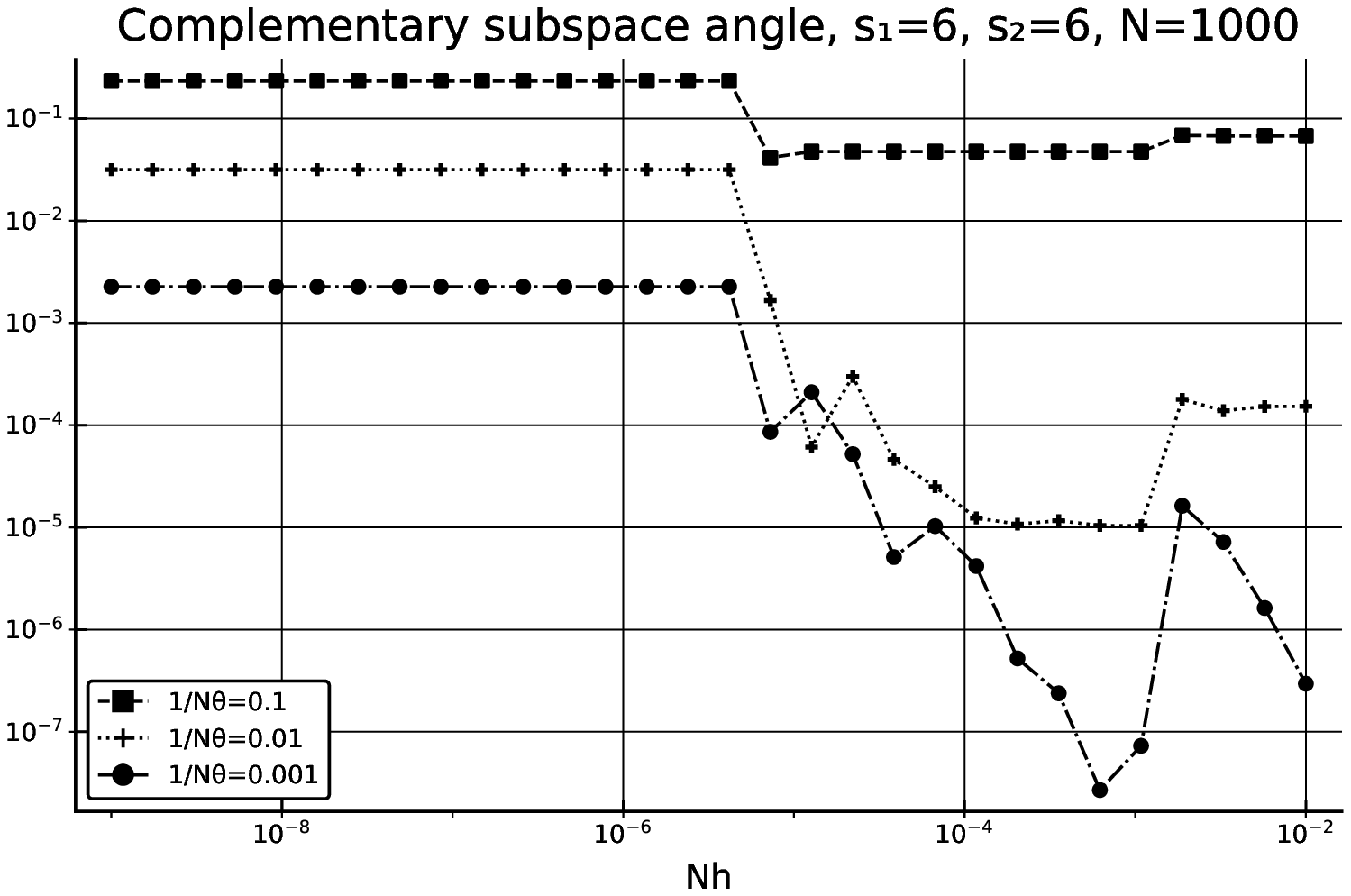}
  \caption{\small Complementary subspace angle $\beta$. $N,\theta$
    fixed, varying $h$.}
  \label{fig:angle-Ntheta}
\end{figure}

\subsection{Spectra of multi-cluster Vandermonde matrices}

We normalize the Vandermonde matrix and compute the spectra of the
matrices ${1\over{\sqrt N}}\VV_N(\xv)$. For each experiment we choose
the values of $N$ and $h$ randomly from within a prescribed range. In
the multi-cluster setting, we construct 2 nontrivial clusters of same
size $h^{(1)}=h^{(2)}=h$, and add zero or more well-separated nodes
(so if $s^{(1)}=2$, $s^{(2)}=3$ and $s=7$, there are 2 clusters of
multiplicity 1).  The values of $\sigma_j({1\over\sqrt{N}}\VV_N(\xv))$
for different $\xv$ are plotted in Figure \ref{fig:singular-values}.

There is good agreement with Theorem \ref{thm:single.cl} and Corollary
\ref{cor:full-sing-vals}. The proportionality constant is seen to be
not too large. Furthermore, the minimal value of $Nh$ for which the
bounds hold (corresponding to the constants $\Cr{single.cluster.Nh}$
and $\Cr{multi.cluster.N.h}$ in Theorems \ref{thm:single.cl} and
\ref{thm.union}, respectively) is apparently reasonably high.

\begin{figure}
  \centering
  \includegraphics[width=0.45\linewidth]{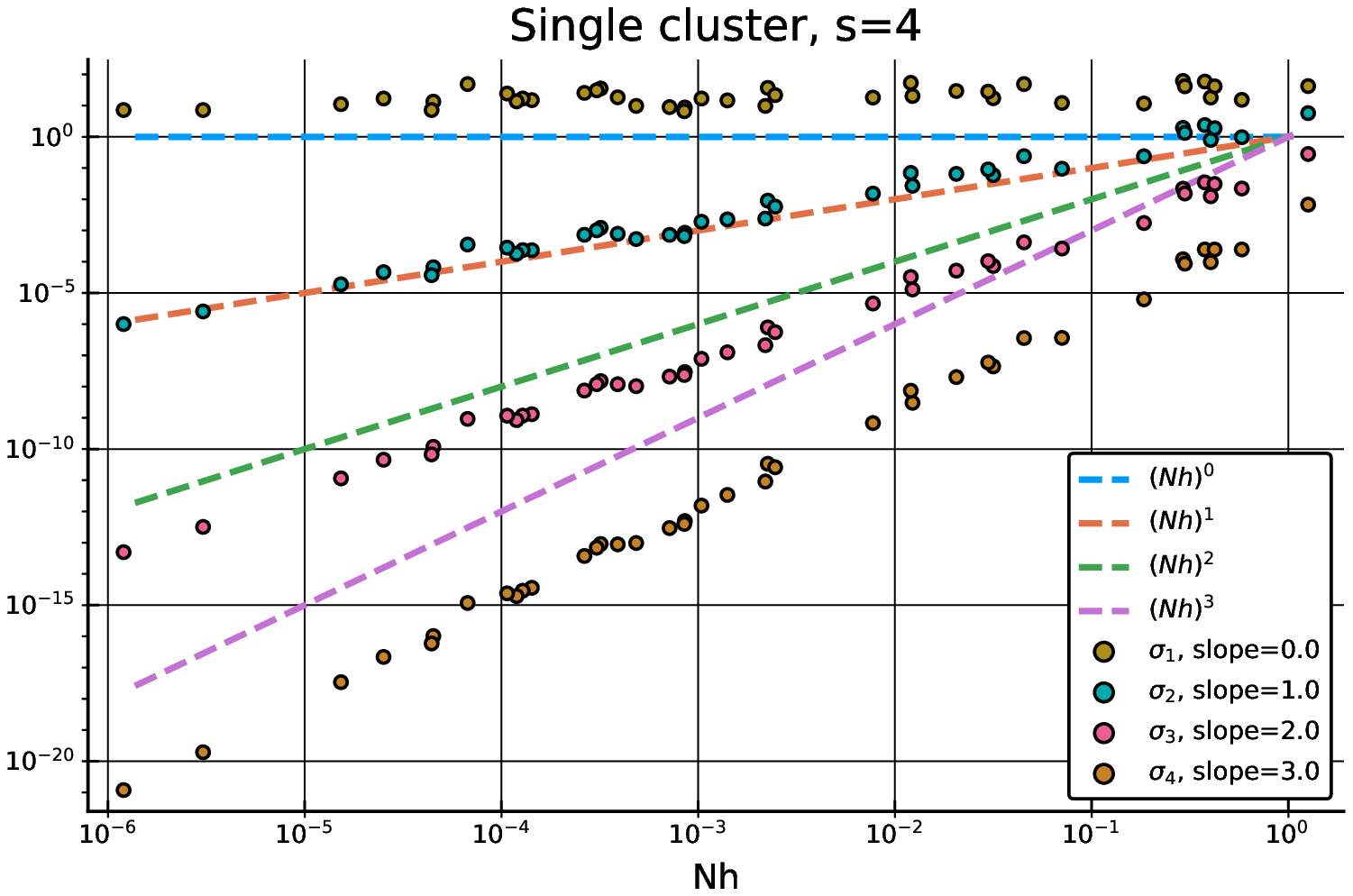}   \includegraphics[width=0.45\linewidth]{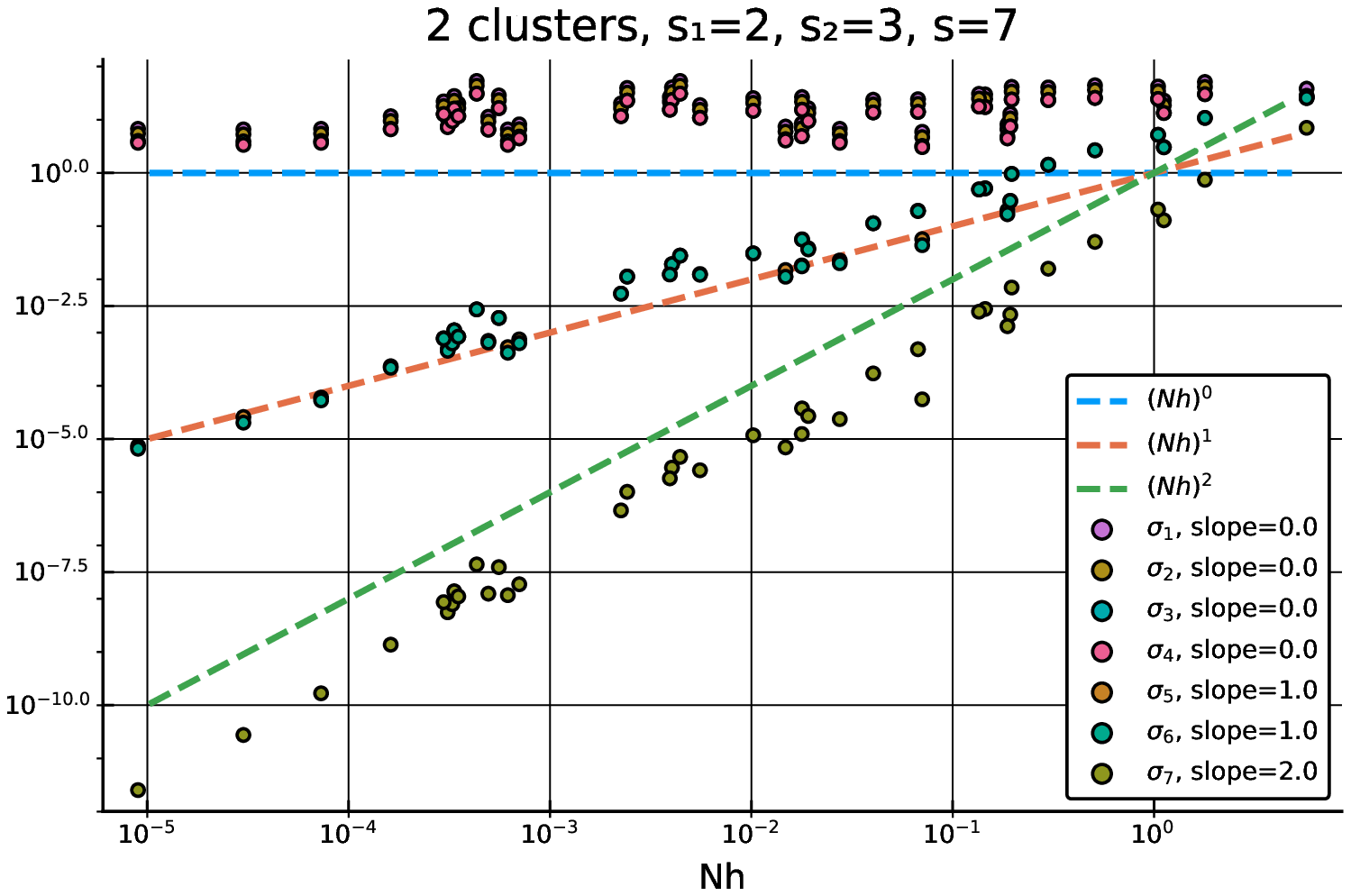}

  \caption{\small Singular values of ${1\over\sqrt{N}}\VV_N$ as a function of
    $Nh$. {\bf (left)} A single cluster with $s=4$. {\bf (right)} 2
    nontrivial clusters (and 4 overall), multiplicities =
    $2,1,3,1$. The slopes of the lines, computed by a linear fit, are
    written in the respective legend labels.}
  \label{fig:singular-values}
\end{figure}
	
\subsection{Least squares accuracy}
Here we solve the least squares problem as in Theorem
\ref{thm:ls-accuracy}. In each experiment, we choose $N,h,\varepsilon$
uniformly at random within prescribed ranges. This in particular
defines $\xv$ and $\VV_N(\xv)$. We then choose the
entries of the vectors $\av_0$ and $\vec{f}$ to be uniformly randomly
distributed in $\left[0,1\right]$. Then we put $\bv_0=\VV_N(\xv)\av_0$
and $\bv = \bv_0 + \varepsilon\vec{f}$. We then compute $\av=\av(\xv,\bv)$
as in Definition \ref{def.ls.solution}. Finally, we set
$$
\delta a_{\ell}:=\frac{|(\av-\av_0)_{\ell}|}{\|\bv-\bv_0\|_{\infty}}.
$$
We then repeat the experiment multiple times, and plot
$\delta a_{\ell}$ for all $\ell=1,\dots,s$ as a function of $Nh$. The
results are presented in Figure \ref{fig:ls-double-s6}. There is a
good agreement with the estimate \eqref{eq:ls-componentwise-bound} for
each component.
\clearpage
\begin{figure}
  \centering
  \includegraphics[width=0.75\linewidth]{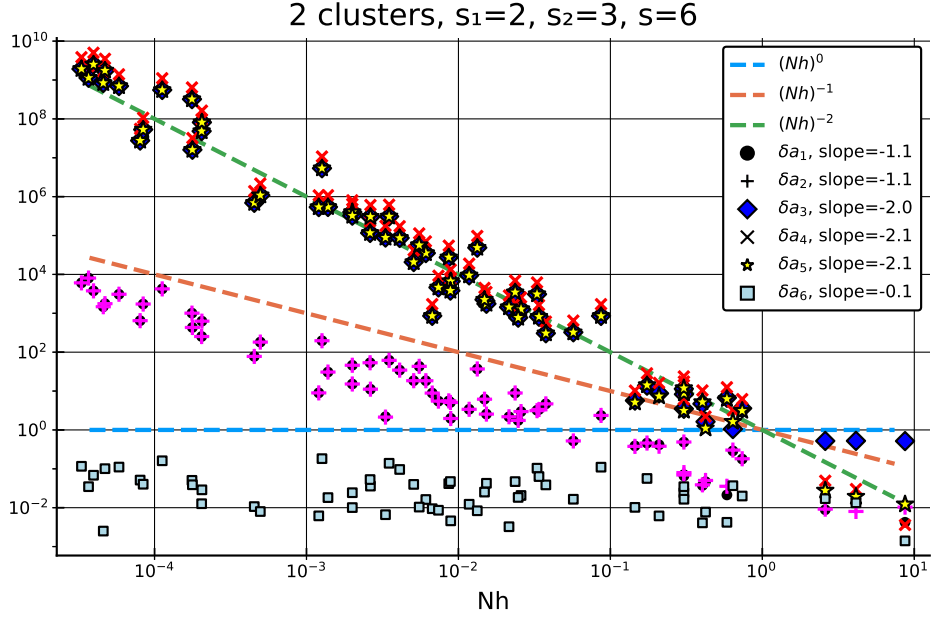}
  \caption{\small Accuracy of least squares reconstruction. 2
    nontrivial clusters (and 3 overall), multiplicities = $2,3,1$. The
    slopes of the lines, computed by a linear fit, are written in the
    respective legend labels.}
  \label{fig:ls-double-s6}
\end{figure}

\appendix

\section{Divided differences}\label{appen.dd}

Recall Definition \ref{def.divided.diff}. The following properties can
be found in e.g. \cite{de2005divided}, see also \cite[Section
6.2]{batenkov_geometry_2014}.

\begin{lemma}\label{lem:dd-properties}
  The functionals $[t_1,\dots,t_n]$ satisfy the following.
  \begin{enumerate}
  \item $[t_1,\dots,t_n]f$ is a symmetric function of
    $(t_1,\dots,t_n)$. 
  \item $[t_1,\dots,t_n]f$ is a continuous function of
    $(t_1,\dots,t_n)$, i.e.
    \begin{equation}\label{eq:dd-continuity}
    \lim_{\left(t_1,\dots,t_n\right)\to(u_1,\dots,u_n)}[t_1,\dots,t_n]f
    = [u_1,\dots,u_n]f.
    \end{equation}
  \item The numbers $[t_1,\dots,t_n]f$ can be computed by the
    recursive rule
    \begin{equation}
      \label{eq:dd-recursive-comp}
      [t_1,\dots,t_n]f =
      \begin{cases}
        f(t_1) & n=1; \\
        \frac{[t_2,\dots,t_n]f - [t_1,\dots,t_{n-1}]f}{t_n-t_1} & t_1 \neq t_n; \\
        \lim_{\xi\to t_n}\bigl\{ \frac{\dt}{\dt \xi}
        \left([\xi,t_2,\dots,t_{n-1}]f\right) \bigr\} & t_1=t_n.
      \end{cases}
    \end{equation}
  \item In particular, if all $\{t_j\}$'s are distinct, then
    \begin{equation}\label{eq:dd-distinct-explicit}
      [t_1,\dots,t_n]f = \sum_{j=1}^n \frac{f(t_j)}{\prod_{k\neq j}(t_j-t_k)}.
    \end{equation}
  \item (Mean value theorem) Let $t_1,\dots,t_n \in \RR$ and put
    $I:=\left[\min_{\ell} t_{\ell}, \max_{\ell} t_{\ell}\right]$. Then
    \begin{equation}\label{eq:dd-mean-value}
      [t_1,\dots,t_n]f = \frac{f^{(n-1)}(\xi)}{(n-1)!},\quad \xi \in I.
    \end{equation}
  \item From the above, in particular,
    \begin{equation}\label{eq:dd-repeated-all}
      [\underbrace{t,t,\dots,t}_{n\text{ times}}]f = \frac{f^{(n-1)}(t)}{(n-1)!}.
    \end{equation}
  \end{enumerate}
\end{lemma}

\section{Power sums}\label{appen.power.sums}
\begin{lemma}
  For a positive integer $p$, the sum of the $p^{th}$ powers of the
  first $N+1$ non-negative integers is given by Faulhaber's formula
  \begin{equation}\label{eq.Faulhaber.formula} \sum_{k=0}^N k^{p} =
    \frac{N^{p+1}}{p+1}+\frac{1}{2}N^p+\sum_{k=2}^p
    \frac{B_{k}}{k!}(p)_{k-1}N^{p-k+1}=\frac{N^{p+1}}{p+1} + \MO(N^{p}),
  \end{equation}
  where $B_k$ are the Bernoulli numbers, and $(p)_{k-1}$ is the
  falling factorial, $(p)_{k-1}=\dfrac{p!}{(p-k+1)!}$.
  We also have the following non-asymptotic bounds:
  \begin{equation}\label{eq.integer.power.sum}
    \frac{N^{p+1}}{p+1}= \int_{0}^N x^p \dt x \le \sum_{k=0}^N k^p \le N^{p+1}.  
  \end{equation}
\end{lemma}
    
\section{Trigonometric
  cancellation}\label{appen.trigonometric.cancellation}

\begin{lemma}\label{lem:norm2}
	For each $z \ne 1$ with $|z|=1$ and for each $m\in \mathbb N$ we have
	\begin{equation}\label{eq:sum3}
		\left|\sum_{k=0}^N k^m z^k\right|\le \frac{2}{|1-z|}N^m.
	\end{equation}
\end{lemma}

\begin{proof}
	Let us notice that a ``naive'' upper bound
	$$
		\left|\sum_{k=0}^N k^m z^k\right|\le \sum_{k=0}^N k^m \sim N^{m+1}
	$$
	is not sufficient for our purposes. To get the order of $N^m$ 
	we have to take into account cancellations in the sum (\ref{eq:sum3}) as follows:
	\begin{align*}
		(1-z)\sum_{k=0}^N k^m z^k=  -N^{m}z^{N+1} + \sum_{k=0}^N k^m z^k -\sum_{k=0}^{N-1} k^m z^{k+1}=  -N^{m}z^{N+1} +\sum_{k=1}^N (k^m-(k-1)^m)z^k. 
	\end{align*}
	Then by the triangle inequality 
	\begin{align*}
		|1-z|\left|\sum_{k=0}^N k^m z^k\right|\le N^{m} +\sum_{k=1}^N k^m-(k-1)^m = 2N^{m}. 
	\end{align*}
	\smallskip
	
	This completes the proof of Lemma \ref{lem:norm2}.
      \end{proof}

\bibliographystyle{abbrv}
\bibliography{bib}	

\begin{thebibliography}{10}

\bibitem{akinshin_accuracy_2015}
A.~Akinshin, D.~Batenkov, and Y.~Yomdin.
\newblock Accuracy of spike-train {{Fourier}} reconstruction for colliding
  nodes.
\newblock In {\em 2015 {{International Conference}} on {{Sampling Theory}} and
  {{Applications}} ({{SampTA}})}, pages 617--621, May 2015.

\bibitem{akinshin_error_2017}
A.~Akinshin, G.~Goldman, and Y.~Yomdin.
\newblock Geometry of error amplification in solving {{Prony}} system with
  near-colliding nodes.
\newblock {\em arXiv:1701.04058 [math]}, Jan. 2017.

\bibitem{aubel_vandermonde_2017}
C.~Aubel and H.~B\"olcskei.
\newblock Vandermonde matrices with nodes in the unit disk and the large sieve.
\newblock {\em Applied and Computational Harmonic Analysis}, Aug. 2017.

\bibitem{batenkov_stability_2016}
D.~Batenkov.
\newblock Stability and super-resolution of generalized spike recovery.
\newblock {\em Applied and Computational Harmonic Analysis}, 45(2):299--323,
  Sept. 2018.

\bibitem{bhandari2019}
D.~Batenkov, A.~Bhandari, and T.~Blu.
\newblock Rethinking {{Super}}-resolution: The {{Bandwidth Selection Problem}}.
\newblock In {\em {{ICASSP}} 2019 - 2019 {{IEEE International Conference}} on
  {{Acoustics}}, {{Speech}} and {{Signal Processing}} ({{ICASSP}})}, pages
  5087--5091, May 2019.

\bibitem{batenkov2018a}
D.~Batenkov, L.~Demanet, G.~Goldman, and Y.~Yomdin.
\newblock Conditioning of {{Partial Nonuniform Fourier Matrices}} with
  {{Clustered Nodes}}.
\newblock {\em SIAM Journal on Matrix Analysis and Applications},
  44(1):199--220, Jan. 2020.

\bibitem{batenkov2019a}
D.~Batenkov, G.~Goldman, and Y.~Yomdin.
\newblock Super-resolution of near-colliding point sources.
\newblock {\em To appear in Information and Inference}.

\bibitem{batenkov_geometry_2014}
D.~Batenkov and Y.~Yomdin.
\newblock Geometry and {{Singularities}} of the {{Prony}} mapping.
\newblock {\em Journal of Singularities}, 10:1--25, 2014.

\bibitem{bazan_conditioning_2000}
F.~Baz\'an.
\newblock Conditioning of rectangular {{Vandermonde}} matrices with nodes in
  the unit disk.
\newblock {\em SIAM Journal on Matrix Analysis and Applications}, 21:679, 2000.

\bibitem{beckermann2000}
B.~Beckermann.
\newblock The condition number of real {{Vandermonde}}, {{Krylov}} and positive
  definite {{Hankel}} matrices.
\newblock {\em Numerische Mathematik}, 85(4):553--577, 2000.

\bibitem{beckermann_sensitivity_1999}
B.~Beckermann and E.~B. Saff.
\newblock The sensitivity of least squares polynomial approximation.
\newblock In {\em Applications and {{Computation}} of {{Orthogonal
  Polynomials}}}, pages 1--19. {Springer}, 1999.

\bibitem{beckermann2019}
B.~Beckermann and A.~Townsend.
\newblock Bounds on the {{Singular Values}} of {{Matrices}} with {{Displacement
  Structure}}.
\newblock {\em SIAM Review}, 61(2):319--344, Jan. 2019.

\bibitem{ben-israel2003}
A.~{Ben-Israel} and T.~N.~E. Greville.
\newblock {\em Generalized Inverses: Theory and Applications}.
\newblock Number~15 in {{CMS}} Books in Mathematics. {Springer}, {New York},
  2nd ed edition, 2003.

\bibitem{bjorck1991}
{\AA}.~Bj{\"o}rck.
\newblock Component-wise perturbation analysis and error bounds for linear
  least squares solutions.
\newblock {\em BIT}, 31(2):237--244, 1991.

\bibitem{bjorck1973}
{\AA}.~Bj{\"o}rck and G.~H. Golub.
\newblock Numerical methods for computing angles between linear subspaces.
\newblock {\em Mathematics of computation}, 27(123):579--594, 1973.

\bibitem{choi1983tricks}
M.-D. Choi.
\newblock Tricks or treats with the {{Hilbert}} matrix.
\newblock {\em The American Mathematical Monthly}, 90(5):301--312, 1983.

\bibitem{crdova1990}
A.~Cordova, W.~Gautschi, and S.~Ruscheweyh.
\newblock Vandermonde matrices on the circle: {{Spectral}} properties and
  conditioning.
\newblock {\em Numerische Mathematik}, 57(1):577--591, Dec. 1990.

\bibitem{de2005divided}
C.~{deBoor}.
\newblock Divided differences.
\newblock {\em Surveys in Approximation Theory}, 1:46--69, 2005.
\newblock [Online article at] http://www.math.technion.ac.il/sat.

\bibitem{demanet_recoverability_2014}
L.~Demanet and N.~Nguyen.
\newblock The recoverability limit for superresolution via sparsity.
\newblock 2014.

\bibitem{diederichs2019}
B.~Diederichs.
\newblock Well-{{Posedness}} of {{Sparse Frequency Estimation}}.
\newblock {\em arXiv:1905.08005 [math]}, May 2019.

\bibitem{donoho_superresolution_1992}
D.~Donoho.
\newblock Superresolution via sparsity constraints.
\newblock {\em SIAM Journal on Mathematical Analysis}, 23(5):1309--1331, 1992.

\bibitem{eisinberg2001}
A.~Eisinberg, P.~Pugliese, and N.~Salerno.
\newblock Vandermonde matrices on integer nodes: The rectangular case.
\newblock {\em Numerische Mathematik}, 87(4):663--674, Feb. 2001.

\bibitem{gautschi_inverses_1962}
W.~Gautschi.
\newblock On inverses of {{Vandermonde}} and confluent {{Vandermonde}}
  matrices.
\newblock {\em Numerische Mathematik}, 4(1):117--123, 1962.

\bibitem{gautschi_inverses_1963}
W.~Gautschi.
\newblock On inverses of {{Vandermonde}} and confluent {{Vandermonde}}
  matrices. {{II}}.
\newblock {\em Numerische Mathematik}, 5(1):425--430, 1963.

\bibitem{gautschi_norm_1974}
W.~Gautschi.
\newblock Norm estimates for inverses of {{Vandermonde}} matrices.
\newblock {\em Numerische Mathematik}, 23(4):337--347, 1974.

\bibitem{gautschi_inverses_1978}
W.~Gautschi.
\newblock On inverses of {{Vandermonde}} and confluent {{Vandermonde}} matrices
  {{III}}.
\newblock {\em Numerische Mathematik}, 29(4):445--450, 1978.

\bibitem{higham_survey_1994}
N.~J. Higham.
\newblock A survey of componentwise perturbation theory in numerical linear
  algebra.
\newblock In {\em Proceedings of Symposia in Applied Mathematics}, volume~48,
  pages 49--77, 1994.

\bibitem{hogan_duration_2011}
J.~A. Hogan and J.~D. Lakey.
\newblock {\em Duration and {{Bandwidth Limiting}}: {{Prolate Functions}},
  {{Sampling}}, and {{Applications}}}.
\newblock {Springer Science \& Business Media}, Dec. 2011.

\bibitem{horn_matrix_2012}
R.~A. Horn and C.~R. Johnson.
\newblock {\em Matrix Analysis}.
\newblock {Cambridge University Press}, {Cambridge ; New York}, 2nd ed edition,
  2012.

\bibitem{knyazev2002}
A.~V. Knyazev and M.~E. Argentati.
\newblock Principal angles between subspaces in an {{A}}-based scalar product:
  Algorithms and perturbation estimates.
\newblock {\em SIAM Journal on Scientific Computing}, 23(6):2008--2040, 2002.

\bibitem{kunis2018}
S.~Kunis and D.~Nagel.
\newblock On the condition number of {{Vandermonde}} matrices with pairs of
  nearly-colliding nodes.
\newblock {\em arXiv:1812.08645 [math]}, Dec. 2018.

\bibitem{kunis2019a}
S.~Kunis and D.~Nagel.
\newblock On the smallest singular value of multivariate {{Vandermonde}}
  matrices with clustered nodes.
\newblock {\em arXiv:1907.07119 [cs, math]}, July 2019.

\bibitem{lee1992}
H.~B. Lee.
\newblock Eigenvalues and eigenvectors of covariance matrices for signals
  closely spaced in frequency.
\newblock {\em IEEE Transactions on signal processing}, 40(10):2518--2535,
  1992.

\bibitem{li_stable_2017}
W.~Li and W.~Liao.
\newblock Stable super-resolution limit and smallest singular value of
  restricted {{Fourier}} matrices.
\newblock {\em arXiv:1709.03146 [cs, math]}, Sept. 2017.

\bibitem{li2019}
W.~Li, W.~Liao, and A.~Fannjiang.
\newblock Super-resolution limit of the {{ESPRIT}} algorithm.
\newblock {\em IEEE Transactions on Information Theory}, pages 1--1, 2020.
\newblock Conference Name: IEEE Transactions on Information Theory.

\bibitem{micchelli1986}
C.~A. Micchelli.
\newblock Interpolation of scattered data: {{Distance}} matrices and
  conditionally positive definite functions.
\newblock {\em Constructive Approximation}, 2(1):11--22, Dec. 1986.

\bibitem{moitra_super-resolution_2015}
A.~Moitra.
\newblock Super-resolution, {{Extremal Functions}} and the {{Condition Number}}
  of {{Vandermonde Matrices}}.
\newblock In {\em Proceedings of the {{Forty}}-{{Seventh Annual ACM}} on
  {{Symposium}} on {{Theory}} of {{Computing}}}, STOC '15, pages 821--830, New
  York, NY, USA, 2015. {ACM}.

\bibitem{pan2016}
V.~Y. Pan.
\newblock How {{Bad Are Vandermonde Matrices}}?
\newblock {\em SIAM Journal on Matrix Analysis and Applications},
  37(2):676--694, Jan. 2016.

\bibitem{slepian1968}
D.~Slepian.
\newblock A numerical method for determining the eigenvalues and eigenfunctions
  of analytic kernels.
\newblock {\em SIAM Journal on Numerical Analysis}, 5(3):586--600, 1968.

\bibitem{slepian_prolate_1978}
D.~Slepian.
\newblock Prolate spheroidal wave functions, fourier analysis, and uncertainty
  -- {{V}}: The discrete case.
\newblock {\em Bell System Technical Journal, The}, 57(5):1371--1430, May 1978.

\bibitem{stewart1990matrix}
G.~W. Stewart.
\newblock {\em Matrix perturbation theory}.
\newblock Academic Press, 1990.

\bibitem{stoica_spectral_2005}
P.~Stoica and R.~Moses.
\newblock {\em Spectral Analysis of Signals}.
\newblock {Pearson/Prentice Hall}, 2005.

\bibitem{todd1954condition}
J.~Todd.
\newblock The condition of the finite segments of the {{Hilbert}} matrix.
\newblock {\em Contributions to the solution of systems of linear equations and
  the determination of eigenvalues}, 39:109--116, 1954.

\bibitem{tuncer_classical_2009}
T.~E. Tuncer and B.~Friedlander.
\newblock {\em Classical and Modern Direction-of-Arrival Estimation}.
\newblock {Academic}, {London}, 2009.
\newblock OCLC: 299711034.

\bibitem{tyrtyshnikov1994}
E.~E. Tyrtyshnikov.
\newblock How bad are {{Hankel}} matrices?
\newblock {\em Numerische Mathematik}, 67(2):261--269, Mar. 1994.

\bibitem{wathen2015}
A.~J. Wathen and S.~Zhu.
\newblock On spectral distribution of kernel matrices related to radial basis
  functions.
\newblock {\em Numerical Algorithms}, 70(4):709--726, Dec. 2015.

\bibitem{wilf2012finite}
H.~S. Wilf.
\newblock {\em Finite sections of some classical inequalities}, volume~52.
\newblock Springer Science \& Business Media, 2012.

\bibitem{yang2018}
Z.~Yang, J.~Li, P.~Stoica, and L.~Xie.
\newblock Chapter 11 - {{Sparse}} methods for direction-of-arrival estimation.
\newblock In R.~Chellappa and S.~Theodoridis, editors, {\em Academic {{Press
  Library}} in {{Signal Processing}}, {{Volume}} 7}, pages 509--581. {Academic
  Press}, Jan. 2018.

\end{thebibliography}
\end{document}